\documentclass[a4paper, reqno]{amsart}
\usepackage{graphicx} 
\usepackage[utf8]{inputenc}

\title{Decomposing multipersistence modules using functor calculus}
\author{Bjørnar Gullikstad Hem}
\thanks{Email: bjornar.hem@epfl.ch }

\usepackage{amsmath}
\usepackage{amsthm}
\usepackage{amssymb}
\usepackage[hidelinks,colorlinks=true]{hyperref}
\hypersetup{allcolors=[rgb]{0,0.31,0.62}}
\usepackage{thmtools}   
\usepackage{mathrsfs}
\usepackage{enumitem}
\usepackage{mathtools}
\usepackage{caption}
\usepackage{subcaption}

\usepackage{tikz}
\usetikzlibrary{matrix} 
\newcommand{\diagram}[3]{\matrix (#1) [matrix of math nodes,row
  sep={#2},column sep={#3},text height=1.5ex,text
  depth=0.25ex]}
\usepackage{tikz-cd}

\theoremstyle{plain}
\newtheorem{theorem}{Theorem}[section]
\newtheorem{proposition}[theorem]{Proposition}
\newtheorem{corollary}[theorem]{Corollary}
\newtheorem{lemma}[theorem]{Lemma}

\theoremstyle{definition}
\newtheorem{definition}[theorem]{Definition}
\newtheorem{example}[theorem]{Example}
\newtheorem{remark}[theorem]{Remark}

\DeclareMathOperator{\id}{id}
\DeclareMathOperator{\colim}{colim}

\DeclareMathOperator{\Ima}{im}

\newcommand{\R}{\mathbb{R}} 

\newcommand{\Fb}{\mathbb{Fb}}

\newcommand{\A}{\mathscr{A}}
\newcommand{\B}{\mathscr{B}}
\newcommand{\C}{\mathscr{C}}
\newcommand{\M}{\mathscr{M}}

\newcommand{\Vecs}{\mathsf{Vec}}

\newcommand{\Po}{\mathcal{P}}
\newcommand{\X}{\mathcal{X}}

\newcommand{\Ch}{\mathrm{Ch}}
\newcommand{\Nn}{\mathbb{N}}

\DeclareMathOperator{\cofib}{cofib}
\DeclareMathOperator{\fib}{fib}
\DeclareMathOperator{\holim}{holim}
\DeclareMathOperator{\hocolim}{hocolim}

\DeclareMathOperator{\coker}{coker}
\DeclareMathOperator{\Fun}{Fun}

\DeclareMathOperator{\Lan}{Lan}
\DeclareMathOperator{\Ran}{Ran}
\newcommand{\projPvec}{\Fun_{\textrm{proj}}(P, \Vecs_{\Fb})}
\DeclareMathOperator{\jdim}{jdim}
\DeclareMathOperator{\mdim}{mdim}

\numberwithin{equation}{subsection}

\newtheorem*{proposition*}{Proposition}
\newtheorem*{theorem*}{Theorem}

\newsavebox{\pullback}
\sbox\pullback{%
\begin{tikzpicture}%
\draw (0,1ex) -- (1ex,1ex);%
\draw (1ex,0ex) -- (1ex,1ex);%
\end{tikzpicture}}

\begin{document}

\maketitle

\begin{abstract}
We apply \emph{poset cocalculus}, a functor calculus framework for functors out of a poset, to study the problem of decomposing multipersistence modules into simpler components. We both prove new results in this topic and offer a new perspective on already established results.
In particular, we show that a pointwise finite-dimensional bipersistence module is middle-exact if and only if it is isomorphic to the homology of a homotopy degree 1 functor, from which we deduce a novel, more synthetic proof of the interval decomposability of middle-exact bipersistence modules.
We also give a new decomposition theorem for middle-exact multipersistence modules indexed over a finite poset, stating that such a module can always be written as a direct sum of a projective module, an injective module, and a bidegree 1 module, even in the case where it is not pointwise finite-dimensional.
\end{abstract}

\setcounter{tocdepth}{1}
\tableofcontents

\section{Introduction}

\subsection{Background}

A \emph{persistence module} is a functor $F \colon P \to \Vecs_{\Fb}$, where $P$ is a poset. We will use the word \emph{multipersistence module} when referring to a persistence module whose source poset is a product of two or more total orders, and the word \emph{bipersistence module} when the source poset is a product of exactly two total orders.
A multipersistence module from a product of $n$ total orders is called a $n$-parameter multipersistence module.

Persistence modules are the main objects in persistent homology, a key technique in the field of topological data analysis (TDA).
In persistent homology, one constructs a persistence module from data such as a point cloud.
From the persistence module, one then typically constructs a visual representation known as the \emph{barcode diagram}. Construction of the barcode diagram is possible due to the structure theorem \cite{interval_decomp_source, structure_theorem_crawley, structure_theorem_webb},
which says that a persistence module can be decomposed into a direct sum of simple components known as \emph{interval modules}.
However, the structure theorem applies only when the source poset is a total order, such as $\R$. Multipersistence modules, on the other hand, are not in general interval decomposable (see, e.g., \cite[Example 8.3]{BotnanLesnickMultipersistence}).

Bipersistence modules are said to be \emph{middle-exact} if they satisfy the following local condition: for every $(x,y) \le (x',y')$, the chain complex
\begin{equation*}
    F(x,y) \to F(x,y') \oplus F(x',y) \to F(x',y')
\end{equation*}
is exact in the middle.
Middle-exact bipersistence modules arise naturally from \emph{interlevel set persistent homology} \cite{interlevelset_persistence_carlsson}, in which one studies the homology of preimages of intervals under a real-valued function on a topological space (see, e.g., \cite[Section 10.2]{BotnanLesnickMultipersistence} for more details).
Interlevel set persistence has been applied to, for example, analysis of digital images \cite{edelsbrunner}.
In \cite{botnanMiddleExactness}, Botnan and Crawley proved that all pointwise finite-dimensional (p.f.d.) middle-exact bipersistence modules are interval decomposable, and that the intervals are of a special form, called \emph{blocks}.

In \cite{lebovici2024localcharacterizationblockdecomposabilitymultiparameter},
the notion of middle exactness is generalized from bipersistence modules to arbitrary multipersistence modules, giving rise to the notions of \emph{$k$-middle-exact modules} for $k \ge 2$.
The authors show that p.f.d. multipersistence modules that are $k$-middle-exact for all $k \ge 2$ are always interval decomposable.

In \cite{hem2025posetfunctorcocalculusapplications}, we introduced \emph{poset cocalculus}, a novel framework for studying functors out of a poset.
Given a functor $F$ from a distributive lattice to a (sufficiently nice) model category, poset cocalculus produces a telescope of codegree $n$ approximations of the functor,
\begin{equation*}
\begin{aligned}
\begin{tikzpicture}
    \diagram{d}{3em}{3em}{
        \vdots & \ \\
        T_1 F & \ \\
        T_0 F & F. \\
    };
    
    \path[->,font = \scriptsize, midway]
    (d-2-1) edge (d-1-1)
    (d-1-1) edge (d-3-2)
    (d-3-1) edge (d-2-1)
    (d-2-1) edge (d-3-2)
    (d-3-1) edge (d-3-2);
\end{tikzpicture}
\end{aligned}
\end{equation*}
where a functor is codegree $n$  if it takes strongly bicartesian $(n+1)$-cubes to homotopy cocartesian $(n+1)$-cubes.
We also proved that taking the codegree $n$ approximation of a persistence module is stable under a certain generalized interleaving distance.
Furthermore, we defined a dual, \emph{poset calculus}, that gives a Taylor \emph{tower} of \emph{degree $n$ approximations}, denoted $T^n F$.

\subsection{Main contributions}

Relations between middle exactness and functor calculus have been explored previously in \cite[Chapter 5.7]{lerch_thesis}.
In \cite{lerch_thesis}, however, the author uses the existing theory of Goodwillie calculus and applies it to functors between $\infty$-categories, and shows that any 1-excisive\footnote{1-excisive is the analogue of ``degree 1'' in Goodwillie calculus.} functor gives rise to a middle-exact functor.
Here, we use the framework of poset cocalculus, which allows us to apply functor calculus concepts directly to functors out of a poset.
One main advantage of this framework is that the universal (co)degree $n$ approximations can be easily computed through an explicit formula established in \cite[Definition 4.21]{hem2025posetfunctorcocalculusapplications}.
Using this approach, we prove several new results, demonstrating the power of poset cocalculus as a tool for analyzing multipersistence modules.

We show that a p.f.d. bipersistence module is interval decomposable if it is either codegree 1 or degree 1.
\begin{proposition*}[\autoref{prop:codegree_1_interval_decomposable}]
    A codegree 1 p.f.d. functor $F \colon [0,\infty]^2 \to \Vecs_{\Fb}$ decomposes into a direct sum of interval modules over death blocks, vertical blocks, and horizontal blocks.
    
    Dually, a degree 1 p.f.d. functor $F \colon [0,\infty]^2 \to \Vecs_{\Fb}$ decomposes into a direct sum of interval modules over birth blocks, vertical blocks, and horizontal blocks.
\end{proposition*}
The definitions of death/birth/vertical/horizontal blocks are given in \autoref{def:blocks}.

We further prove the following theorem on middle-exact bipersistence modules. 
\begin{theorem*}[\autoref{thm:main_theorem}]
    Let $F \colon [0, \infty]^2 \to \Vecs_{\Fb}$ be a p.f.d. functor. The following are equivalent.
    \begin{description}
        \item[(i)] $F$ is middle-exact.
        \item[(ii)] $F$ is the direct sum of a degree 1 functor and a codegree 1 functor.
        \item[(iii)] There exists a homotopy degree 1 functor $\widehat F \colon [0, \infty]^2 \to \Ch_{\Fb}$ such that $F \cong H_0 \circ \widehat F$.
    \end{description}
\end{theorem*}
Here, homotopy degree 1 means that it is degree 1 with respect to the canonical model structure on chain complexes over a field. In other words, not only does a homotopy degree 1 functor from $[0,\infty]^2$ gives rise to a middle-exact bipersistence module, but every bipersistence module from $[0, \infty]^2$ arises this way.

A corollary of these two results is that middle-exact p.f.d. persistence modules over $[0, \infty]^2$ are interval decomposable.
Hence, poset cocalculus leads to an alternative, more synthetic proof of the interval decomposability of middle-exact modules.

We further apply poset cocalculus to study higher-dimensional middle exactness. We introduce the terminology \emph{bidegree $n$}, which means ``both degree $n$ and codegree $n$'', and prove the following.
\begin{theorem*}[\autoref{thm:middle_exactnes_general}]
    Let $P$ be a finite product of finite total orders, and let $\Fb$ be a field. If $F \colon P \to \Vecs_{\Fb}$ be a functor that is $k$-middle-exact for all $k \ge 2$, then there exist $B,K,C \colon P \to \Vecs_{\Fb}$ with $K$ injective, $C$ projective, and $B$ bidegree 1, such that $F \cong B \oplus K \oplus C$.
\end{theorem*}
This result applies to all functors from $P$ to $\Vecs_{\Fb}$, even those that are not p.f.d. 
On the other hand, the techniques used in \cite{botnanMiddleExactness} and \cite{lebovici2024localcharacterizationblockdecomposabilitymultiparameter}
rely on the fact that p.f.d. multipersistence modules are uniquely decomposable into indecomposable modules.\footnote{These indecomposable modules need not, however, be interval modules.}
Hence, poset cocalculus offer a novel approach to study multipersistence modules that allows us to understand the structure of modules that are not p.f.d.

To aid our understanding of higher-dimensional multipersistence modules, we also introduce the notion of \emph{layers} and \emph{colayers}.
The $n$th \emph{layer} is defined as $D^n F = \fib(T^n F \to T^{n-1} F)$, and the $n$th \emph{colayer} is the dual concept.
We prove the following proposition.
\begin{proposition*}[\autoref{prop:layers}]
    Let $P$ be a finite product of total orders with minimal elements, $\C$ a complete pointed category, and $F \colon P \to \C$ any functor. Then $D_n F$ is bidegree $n$.

    Dually, let $P$ be a finite product of total orders with maximal elements, $\C$ a cocomplete pointed category, and $F \colon P \to \C$ any functor. Then $D^n F$ is bidegree $n$.
\end{proposition*}
We apply this result to prove that bidegree 1 p.f.d. multipersistence modules are interval decomposable.
\begin{proposition*}[\autoref{prop:bideg1_interval_decomp}]
    Let $P$ be a finite product of finite total orders, $\Fb$ a field, and $F \colon P \to \Vecs_{\Fb}$ a p.f.d. bidegree 1 functor. Then $F$ is interval decomposable.
\end{proposition*}
Together with \autoref{thm:middle_exactnes_general}, this gives an alternative proof of the fact that p.f.d. multipersistence modules that are $k$-middle-exact for every $k \ge 2$ are interval decomposable, as long as the source poset is finite.


\subsection{Notation}

We denote by $\Vecs_{\Fb}$ the category of vector spaces over a field $\Fb$. We denote by $\Ch_{\Fb}$ the model category of unbounded chain complexes over $\Vecs_{\Fb}$, with the canonical model structure \cite[Theorem 2.3.13]{hovey}.

For two posets $P$ and $Q$, the product partial order on $P \times Q$ is given by \begin{equation*}
    (p,q) \le (p',q') \iff p \le q \textrm{ and } p' \le q'.
\end{equation*}
Unless otherwise specified, we endow products of posets with this partial order.
We let $[0, \infty]$ denote the poset of nonnegative numbers, with a maximal element $\infty$ added. We let $[0, \infty]^2$ denote the poset $[0, \infty] \times [0, \infty]$ with the product partial order.

Note that by \cite{Hess_2017}, both the projective and injective model structures on $\Fun(P, \M)$ exist whenever $\M$ is equipped with an accessible model structure and $P$ is a small category. In particular, they exist for $\M = \Ch_{\Fb}$, equipped with the canonical model structure.

Given categories $\A, \B$ and $\C$, and functors $\alpha \colon \A \to \B$ and $F \colon \A \to \C$, we let $\Lan_{\alpha} F \colon \B \to \C$ denote the left Kan extension of $F$ along $\alpha$, whenever it exists. Similarly, we let $\Ran_\alpha F \colon \B \to \C$ denote the right Kan extension of $F$ along $\alpha$, whenever it exists.

\subsection*{Acknowledgments}

I would like to thank my supervisor, Kathryn Hess, whose continued support and advice has been of great value to this work.

\section{Preliminaries on poset cocalculus}

\subsection{Posets and lattices}

Given two elements $x, y$ of a poset, we denote their \emph{least upper bound}, or \emph{join}, by $x \vee y$ and their \emph{greatest lower bound}, or \emph{meet}, by $x \wedge y$ (whenever they exist). A \emph{lattice} is a poset in which $x \vee y$ and $x \wedge y$ exist for all pairs of elements $(x, y)$.

An element $x$ in a poset $P$ is \emph{minimal} if there is no $y \in P$ such that $y < x$. If $P$ is a lattice, then $x \in P$ is minimal if and only if $x$ is a \emph{least element}, i.e., $x \le y$ for all $y \in P$ \cite[Section 4.1]{hem2025posetfunctorcocalculusapplications}.

We similarly say that an element $x$ in a poset $P$ is \emph{maximal} if there is no $y \in P$ such that $y > x$. If $P$ is a lattice, then $x \in P$ is maximal if and only if $x$ is a \emph{greatest element}, i.e., $x \ge y$ for all $y \in P$.

\begin{definition}\label{def:descending_chain}
    A poset $P$ satisfies the \emph{descending chain condition} if every nonempty subset of $P$ has a minimal element.
    
    Dually, a poset $P$ satisfies the \emph{ascending chain condition} if every nonempty subset of $P$ has a maximal element.
\end{definition}

Finite posets trivially satisfy both the descending chain condition and the ascending chain condition.

\subsection{Distributive lattices}

A lattice $P$ is \emph{distributive} if for all elements $x,y,z \in P$,
\begin{equation}\label{eq:dist_lattice}
    x \wedge (y \vee z) = (x \wedge y) \vee (x \wedge z).
\end{equation}
Equivalently, a lattice $P$ is distributive if for all elements $x,y,z \in P$,
\begin{equation}\label{eq:dist_lattice_alt}
    x \vee (y \wedge z) = (x \vee y) \wedge (x \vee z)
\end{equation}
\cite[Chapter IX, Theorem 1]{Birkhoff}.

\begin{definition}
We say that an element $v$ in a lattice $P$ is \emph{join-irreducible} if $v \neq x \vee y$ for all $x,y < v$.
\end{definition}
\begin{definition}
    Let $P$ be a lattice, and let $v \in P$. A \emph{join-decomposition} of $v$ (of size $k$) is a finite collection of elements $p^0, \dots, p^{k-1} \in P$ such that $v = p^0 \vee \dots \vee p^{k-1}$.
\end{definition}
\begin{definition}
We say that a join-decomposition $x^1, \dots, x^k$ is \emph{indecomposable} if it consists only of join-irreducible elements.
\end{definition}
\begin{definition}
    Let $p^0, \dots, p^{k-1}$ be a join-decomposition of $v$. We say that the join-decomposition is \emph{reduced} if none of the $p^i$'s is redundant, i.e., if for all $i$,
    \begin{equation*}
        \bigvee_{j \in [k-1] \setminus \{i\}} p_j \neq v.
    \end{equation*}
    We adopt the convention that $p^0 = v$ is a reduced join-decomposition of $v$, unless $v$ is minimal.
\end{definition}
\begin{lemma}{\cite[Lemma 1, p.\ 142]{Birkhoff}}
    Let $P$ be a distributive lattice. If $v \in P$ has an indecomposable reduced join-decomposition, it is unique.
\end{lemma}
\begin{definition}\label{def:dim}
Let $P$ be a distributive lattice. For an element $v \in P$, we define the \emph{join-dimension} of $v$ as
\begin{equation*}
    \jdim (v) = 
    \begin{cases}
        0, &\ v \textrm{ is minimal}, \\
        k, &\ \text{$v$ is not minimal and has a reduced,}\\[-4pt]
        \ &\ \textrm{indecomposable join-decomposition of size $k$,} \\
        \infty, &\ \text{otherwise.} \\
    \end{cases}
\end{equation*}
\end{definition}

The following proposition says that in distributive lattices satisfying the descending chain condition, every element has finite join-dimension.
\begin{proposition}{\cite[Theorem 9, p.\ 142]{Birkhoff}}\label{thm:Birkhoff}
    If $P$ is a distributive lattice that satisfies the descending chain condition, then every element of $P$ has a unique indecomposable reduced join-decomposition.
\end{proposition}

\begin{definition}\label{def:join_fact_lattice}
    Let $P$ be a distributive lattice with a minimal element. We say that $P$ is a \emph{join-factorization lattice} if there exists a distributive lattice $Q$, satisfying the descending chain condition, and an order-preserving function $f \colon P \to Q$ such that for all $v \in P$,
    \begin{equation*}
        \jdim(f(v)) = \jdim(v).
    \end{equation*}
\end{definition}
Examples of join-factorization lattices include distributive lattices satisfying the descending chain condition (and, in particular, finite distributive lattices) \cite[Example 4.17]{hem2025posetfunctorcocalculusapplications}.
Furthermore, if $T_1, \dots, T_n$ are total orders with minimal elements, then $P = T_1 \times \dots \times T_n$ is a join-factorization lattice \cite[Example 4.18]{hem2025posetfunctorcocalculusapplications}.

Replacing joins with meets, we dualize the constructions above. In particular, the \emph{meet-dimension} of an element $v$, denoted $\mdim(v)$, is the size of its reduced indecomposable meet-decomposition, if it exists, and is $\infty$ otherwise.

Given a distributive lattice $P$, we let $P_{\le n}$ denote the subposet of elements with join-dimension $\le n$ and $P^{\le n}$ the subposet of elements with meet-dimension $\le n$. In particular, $P_{\le 1}$ consists of the join-irreducible elements and $P^{\le 1}$ consists of the meet-irreducible elements.

\subsection{Cubical diagrams}

Let $[k] = \{0, \dots, k\}$, and let $\Po_{k+1}$ be the power set of $[k]$, viewed as a poset ordered by inclusion. 
For $\C$ any category, we refer to functors $\X \colon \Po_k \to \C$ as \emph{$k$-cubes} in $\C$. 

\begin{definition}
    Let $\C$ be a category.
    A $(k+1)$-cube $\X \colon \Po_{k+1} \to \C$ is \emph{cocartesian} if the canonical map
    \begin{equation*}
        \underset{S \subsetneq [k]}{\colim \X(S)} \to \X([k])
    \end{equation*}
    is an isomorphism.

    Similarly, $\X$ is \emph{cartesian} if the canonical map
    \begin{equation*}
        \X(\emptyset) \to \underset{S \subseteq [k], S \neq \emptyset}{\lim \X(S)}
    \end{equation*}
    is an isomorphism.
\end{definition}

\begin{definition}
    Let $\C$ be a category.
    A $(k+1)$-cube $\X \colon \Po_{k+1} \to \C$ is \emph{strongly cartesian} (resp. \emph{strongly cocartesian}) if each face of dimension at least is cartesian (resp. cocartesian).

    If $\X$ is both strongly cartesian and strongly cocartesian, it is called \emph{strongly bicartesian}.
\end{definition}

\begin{definition}
    Let $\M$ be a model category.
    A $(k+1)$-cube $\X \colon \Po_{k+1} \to \M$ is \emph{homotopy cocartesian} if the canonical map
    \begin{equation}\label{eq:hocolim_XS}
        \underset{S \subsetneq [k]}{\hocolim \X(S)} \to \X([k])
    \end{equation}
    is a weak equivalence.

    Similarly, $\X$ is \emph{homotopy cartesian} if the canonical map
    \begin{equation*}
        \X(\emptyset) \to \underset{S \subseteq [k], S \neq \emptyset}{\holim \X(S)}
    \end{equation*}
    is a weak equivalence.
\end{definition}

We now give a characterization of strongly bicartesian cubes in a distributive lattice. Note that for any finite diagram $\alpha \colon I \to P$ into a lattice, $\colim_I \alpha$ is equal to the join $\bigvee_{i \in I}\alpha(i)$, and dually, $\lim_I {\alpha} = \bigwedge_{i \in I}\alpha(i)$.
\begin{definition}
    Let $P$ be a lattice, let $v \in P$ and let $k$ be a positive integer.
    A \emph{pairwise cover} of $v$ of size $k$ is a collection of elements $x^0, \dots, x^{k-1} \in P$, with $x^i \le v$ for all $i$, such that $x^i \vee x^j = v$ for all $i \neq j$.
\end{definition}
\begin{definition}\label{def:poset_cube}
    Let $P$ be a lattice, and let $v \in P$. Let further $x^0, \dots, x^k$ be a pairwise cover of $v$.
    We define the $(k+1)$-cube 
    \begin{equation*}
        \X_{x^0, \dots, x^k} \colon \Po_{k+1} \to P
    \end{equation*}
    as follows.
    \begin{equation*}
        \X_{x^0, \dots, x^k} (S) = 
        \begin{cases}
            v, &\quad S = [k], \\ 
            \bigwedge_{i \notin S} x^i, &\quad \text{ otherwise.} \\
        \end{cases}
    \end{equation*}
\end{definition}
The following are Lemma 3.6 and Lemma 3.7 in \cite{hem2025posetfunctorcocalculusapplications}.
\begin{lemma}\label{lem:bicartesian_from_codecomp}
    Let $P$ be a distributive lattice. For every $v \in P$ and every pairwise cover $x^0, \dots, x^k$ of $v$, the cube $\X_{x^0, \dots, x^k}$ is strongly bicartesian.
\end{lemma}
\begin{lemma}\label{lem:codecomp_from_bicartesian}
    Let $P$ be a lattice, and let $\X \colon \Po_{k+1} \to P$ be a strongly bicartesian cube. For $i \in [k]$,\footnote{Recall that $[k] = \{0, \dots, k\}.$} let $x^i = \X\big([k] \setminus \{i\}\big)$. Then $x^0, \dots, x^k$ is a pairwise cover of $\X\big([k]\big)$.

    Furthermore, $\X = \X_{x^0, \dots, x^k}$, as defined in \autoref{def:poset_cube}.
\end{lemma}

\subsection{Poset cocalculus}

We recall here the main definitions and results of the \emph{poset cocalculus} introduced in \cite{hem2025posetfunctorcocalculusapplications}.

\begin{definition}
    Let $P$ be a distributive lattice, and $\C$ a finitely cocomplete category. A functor $F \colon P \to \C$ is \emph{codegree $n$} if it takes strongly bicartesian $(n+1)$-cubes to cocartesian $(n+1)$-cubes.
\end{definition}

Given a functor $F \colon P \to \C$, we denote by $T_n F$ its \emph{codegree $n$ approximation} (defined as in \cite[Definition 4.30]{hem2025posetfunctorcocalculusapplications}).
That is, $T_n F$ is the left Kan extension of the restriction $F|_{P_{\le n}}$ along the inclusion $i \colon P_{\le n} \hookrightarrow P$.
This can be computed objectwise as
\begin{equation}\label{eq:T_n_cat}
    T_n F (x) = \underset{v \in P_{\le n}, v \le x}{\colim} F(v).
\end{equation}
Observe that $T_n$ commutes with colimits in $\Fun(P, \C)$, as colimits commute with colimits.

The following are Theorem 4.31 and Theorem 4.32 in \cite{hem2025posetfunctorcocalculusapplications}.
\begin{theorem}\label{theoremA_cat}
    If $P$ is a distributive lattice, then for every functor $F \colon P \to \C$ to a cocomplete category, $T_k F$ is codegree $k$.
\end{theorem}
\begin{theorem}\label{theoremB_cat}
    Let $P$ be a join-factorization lattice, and let $F \colon P \to \C$ be a functor to a cocomplete category. If $F$ is codegree $n$, then $\varepsilon_n \colon T_n F \to F$ is a natural isomorphism.
\end{theorem}

We give analogous definitions for functors to model categories.

\begin{definition}
    Let $P$ be a distributive lattice, and $\M$ a model category such that $\Fun(P, \M)$ admits the projective model structure. A functor $F \colon P \to \M$ is \emph{homotopy codegree $n$} if it takes strongly bicartesian $(n+1)$-cubes to homotopy cocartesian $(n+1)$-cubes.
\end{definition}

When working in the model category setting, we denote the analogue of the codegree $n$ approximation by $\widetilde T_n F$ and call it the \emph{homotopy codegree $n$ approximation} (defined as in \cite[Definition 4.21]{hem2025posetfunctorcocalculusapplications}).
That is, $\widetilde T_n F$ is the homotopy left Kan extension of the restriction $F|_{P_{\le n}}$ along the inclusion $i \colon P_{\le n} \hookrightarrow P$.
This can be computed objectwise as
\begin{equation}\label{eq:htpy_codeg_n_defn}
    \widetilde T_n F (x) = \underset{v \in P_{\le n}, v \le x}{\hocolim} F(v).
\end{equation}

The following are Theorem 4.26 and Theorem 4.27 in \cite{hem2025posetfunctorcocalculusapplications}.
\begin{theorem}\label{theoremA}
    If $P$ is a distributive lattice, then for every functor $F \colon P \to \M$ to a good model category, $\widetilde T_k F$ is homotopy codegree $k$.
\end{theorem}
\begin{theorem}\label{theoremB}
    Let $P$ be a join-factorization lattice, and let $F \colon P \to \M$ be a functor to a good model category. If $F$ is homotopy codegree $n$, then $\varepsilon_n \colon \widetilde T_n F \to F$ is an objectwise weak equivalence.
\end{theorem}
Here, a model category $\M$ is \emph{good} if $\Fun(P, \M)$ admits the projective model structure.

\subsection{Poset calculus}

Poset cocalculus has a dual, \emph{poset calculus}.
We recall the main definitions and results here.

\begin{definition}
    Let $P$ be a distributive lattice, and $\C$ a finitely complete category. A functor $F \colon P \to \C$ is \emph{degree $n$} if it takes strongly bicartesian $(n+1)$-cubes to cartesian $(n+1)$-cubes.
\end{definition}

Given a functor $F \colon P \to \C$, we denote by $T^n F$ its \emph{degree $n$ approximation} (defined as in \cite[Section 8]{hem2025posetfunctorcocalculusapplications}).
Explicitly,
\begin{equation}
    T^n F (x) = \underset{v \in P^{\le n}, v \le x}{\lim} F(v).
\end{equation}
Observe that $T^n$ commutes with limits in $\Fun(P, \C)$, as limits commute with limits.

\begin{definition}
    Let $P$ be a distributive lattice, and $\M$ a model category such that $\Fun(P, \M)$ admits the injective model structure. A functor $F \colon P \to \M$ is \emph{homotopy degree $n$} if it takes strongly bicartesian $(n+1)$-cubes to homotopy cartesian $(n+1)$-cubes.
\end{definition}

As in the dual case, we let $\widetilde T^n F$ denote the \emph{homotopy degree $n$ approximation} of a functor $F$ (defined as in \cite[Section 8]{hem2025posetfunctorcocalculusapplications}).
Explicitly,
\begin{equation}
    \widetilde T^n F (x) = \underset{v \in P^{\le n}, v \le x}{\holim} F(v).
\end{equation}

Dual versions of \autoref{theoremA_cat}, \autoref{theoremB_cat}, \autoref{theoremA} and \autoref{theoremB} hold for poset calculus \cite[Section 8]{hem2025posetfunctorcocalculusapplications}.

\section{Preliminaries on persistence modules}

We briefly state here the main definitions and results in the theory of persistence modules.
We first give the definition of an \emph{interval decomposable} persistence module, and state some established results regarding interval decomposition.
We then go on to state the definition of a middle-exact bipersistence module.

\subsection{Decomposition of persistence modules}

\begin{definition}
    Let $P$ be a poset. An \emph{interval} in $P$ is a subset $I \subset P$ such that the following two conditions hold.
    \begin{itemize}
        \item If $x, y \in I$ and $x \le z \le y$, then $z \in I$.
        \item If $x, y \in I$, then there exists a zigzag of partial order relations
        \begin{equation*}
            x = a_0 \le b_0 \ge a_1 \le \dots \le b_n = y
        \end{equation*}
        connecting $x$ and $y$.
    \end{itemize}
\end{definition}

\begin{definition}
    Let $\Fb$ be a field, let $P$ be a poset and let $I \subset P$ be an interval. The \emph{interval module} $\Fb_{I} \colon P \to \Vecs_{\Fb}$ is defined as
    \begin{equation*}
        \Fb_I(x) = 
        \begin{cases}
            0, \quad x \notin I, \\
            \Fb, \quad x \in I,
        \end{cases}
    \end{equation*}
    and
    \begin{equation*}
        \Fb_I(x \le y) = 
        \begin{cases}
            \id_{\Fb}, \quad x,y \in I, \\
            0, \quad \textrm{otherwise.}
        \end{cases}
    \end{equation*}
\end{definition}

\begin{definition}
    A persistence module $F \colon P \to \Vecs_{\Fb}$ is \emph{interval decomposable} if it is isomorphic to a direct sum of interval modules, i.e., if there exists a set $\mathcal{B}(F)$ of intervals in $P$ such that
    \begin{equation*}
        F \ \cong \ \bigoplus_{I \in \mathcal{B}(F)} \Fb_{I}.
    \end{equation*}
\end{definition}

We say that a persistence module $F \colon P \to \Vecs_{\Fb}$ is \emph{pointwise finite-dimensional} (p.f.d.) if $F(x)$ is finite-dimensional for every $x \in P$.

We say that $F \colon P \to \Vecs_{\Fb}$ is \emph{indecomposable} if $F \cong F' \oplus F''$ implies that $F' \cong 0$ or $F'' \cong 0$.

\begin{theorem}\label{thm:general_decomposition}
    For $P$ any poset and $F \colon P \to \Vecs_{\Fb}$ pointwise finite-dimensional, $F$ is isomorphic to a direct sum of indecomposable persistence modules.
\end{theorem}

\begin{proof}
    This is stated in \cite[Theorem 4.2]{BotnanLesnickMultipersistence}. 
\end{proof}

Note that in the preceding theorem, the indecomposable persistence modules need not be interval modules. This contrasts with the single-parameter case, i.e., when $P = \R$, where every p.f.d. persistence module decomposes into a direct sum of interval modules \cite{interval_decomp_source, structure_theorem_crawley, structure_theorem_webb}.

We say that a subset $Q \subseteq P$ is \emph{directed} if for every $x,y \in Q$, there exists a $u \in Q$ such that $u \ge x,y$. We say that $Q \subseteq P$ is \emph{down-closed} if for every $x \in Q$ and $y \in P$ with $y \le x$, we have $y \in Q$.
Note that a down-closed and directed subset is an interval.

An object $A$ in a category is \emph{injective} if, for any morphism $f \colon X \to A$ and monomorphism $j \colon X \to Y$, there exists a morphism $g \colon Y \to A$ such that $gj = f$.
\begin{lemma}\label{lemma:directed_ideal_injective}
    Let $P$ be a poset, and let $I \subseteq P$ be down-closed and directed. Then
    \begin{equation*}
        \Fb_{I} \colon P \to \Vecs_{\Fb}
    \end{equation*}
    is an injective object in $\Fun(P, \Vecs_{\Fb})$.
\end{lemma}
\begin{proof}
    This is stated in \cite[Lemma 2.1]{botnanMiddleExactness}.
\end{proof}
Note that the dual of this statement is not true. That is, an interval module over an up-closed, codirected interval is not projective in general. For example, the module $\Fb_{(a, \infty)} \colon \R \to \Vecs_{\Fb}$ is not projective (\cite[Lemma 3.2.11]{lerch_thesis}).

\begin{lemma}\label{lemma:axises_decomposable}
    Let $P \subseteq [0, \infty]^2$ be one of:
    \begin{itemize}
        \item $([0, \infty] \times \{0\}) \cup (\{0\} \times [0, \infty])$,
        \item $([0, \infty] \times \{\infty\}) \cup (\{\infty\} \times [0, \infty])$,
    \end{itemize}
    and let $F \colon P \to \Vecs_{\Fb}$ be p.f.d. Then $F$ is interval decomposable.
\end{lemma}
\begin{proof}
    By \cite[Lemma 5.3]{botnanMiddleExactness}, for the posets $P$ of this lemma, the indecomposable persistence modules with source $P$ are precisely the interval modules. The result now follows from \autoref{thm:general_decomposition}.
\end{proof}

\subsection{Middle exactness}

\begin{definition}
    Let $\A$ be an abelian category.
    A commuting square in $\A$
    \begin{center}
    \begin{tikzcd}
    A \arrow[r, "\alpha"] \arrow[d, "f"] & B \arrow[d, "g"] \\
    C \arrow[r, "\beta"]           & D,
    \end{tikzcd}
    \end{center}
    is said to be \emph{middle-exact} if the following sequence, called the associated complex,
    \begin{equation*}
        A \xrightarrow{\begin{pmatrix}
            \alpha \\ f
        \end{pmatrix}} B \oplus C \xrightarrow{\begin{pmatrix}
            g & -\beta
        \end{pmatrix}} D
    \end{equation*}
    is exact (i.e., the image of the left map equals the kernel of the right map).
\end{definition}

The following definition is Definition 5.1 in \cite{botnanMiddleExactness}.

\begin{definition}\label{def:middle_exactness}
    Let $P = S \times T$, where $S$ and $T$ are total orders.
    A multipersistence module $F \colon P \to \Vecs_{\Fb}$ is \emph{middle-exact} if, for all $x, y \in P$, the square
    \begin{center}
    \begin{tikzcd}
    F(x \wedge y) \arrow[r] \arrow[d] & F(x) \arrow[d] \\
    F(y) \arrow[r]           & F(x \vee y)
    \end{tikzcd}
    \end{center}
    is middle-exact.
\end{definition}

\section{(Co)degree 1 bipersistence modules}

In this section, we show that a p.f.d. bipersistence module is interval decomposable if it is degree 1 or codegree 1.

Let $P$ be a poset. Recall that a subset $Q \subseteq P$ is \emph{down-closed} if for every $x \in Q$, whenever $y \in P$ satisfies $y \le x$, then $y \in Q$. Similarly, we say that a subset $Q \subseteq P$ is \emph{up-closed} if for every $x \in Q$, whenever $y \in P$ satisfies $y \ge x$, then $y \in Q$.
\begin{definition}\label{def:blocks}
    An interval $I \subseteq [0, \infty]^2$ is called a \emph{block} if it is one of the following
    \begin{itemize}
        \item \textbf{(Death block)} $I = J_1 \times J_2$ for down-closed subsets $J_1, J_2 \subseteq [0, \infty]$.
        \item \textbf{(Birth block)} $I = J_1 \times J_2$ for up-closed subsets $J_1, J_2 \subseteq [0,\infty]$.
        \item \textbf{(Vertical block)} $I = J \times [0,\infty]$ for an interval $J \subseteq [0,\infty]$.
        \item \textbf{(Horizontal block)} $I = [0,\infty] \times J$ for an interval $J \subseteq [0,\infty]$.
    \end{itemize}
\end{definition}

\begin{figure}[htbp]
    \centering
    \includegraphics[width=0.5\linewidth]{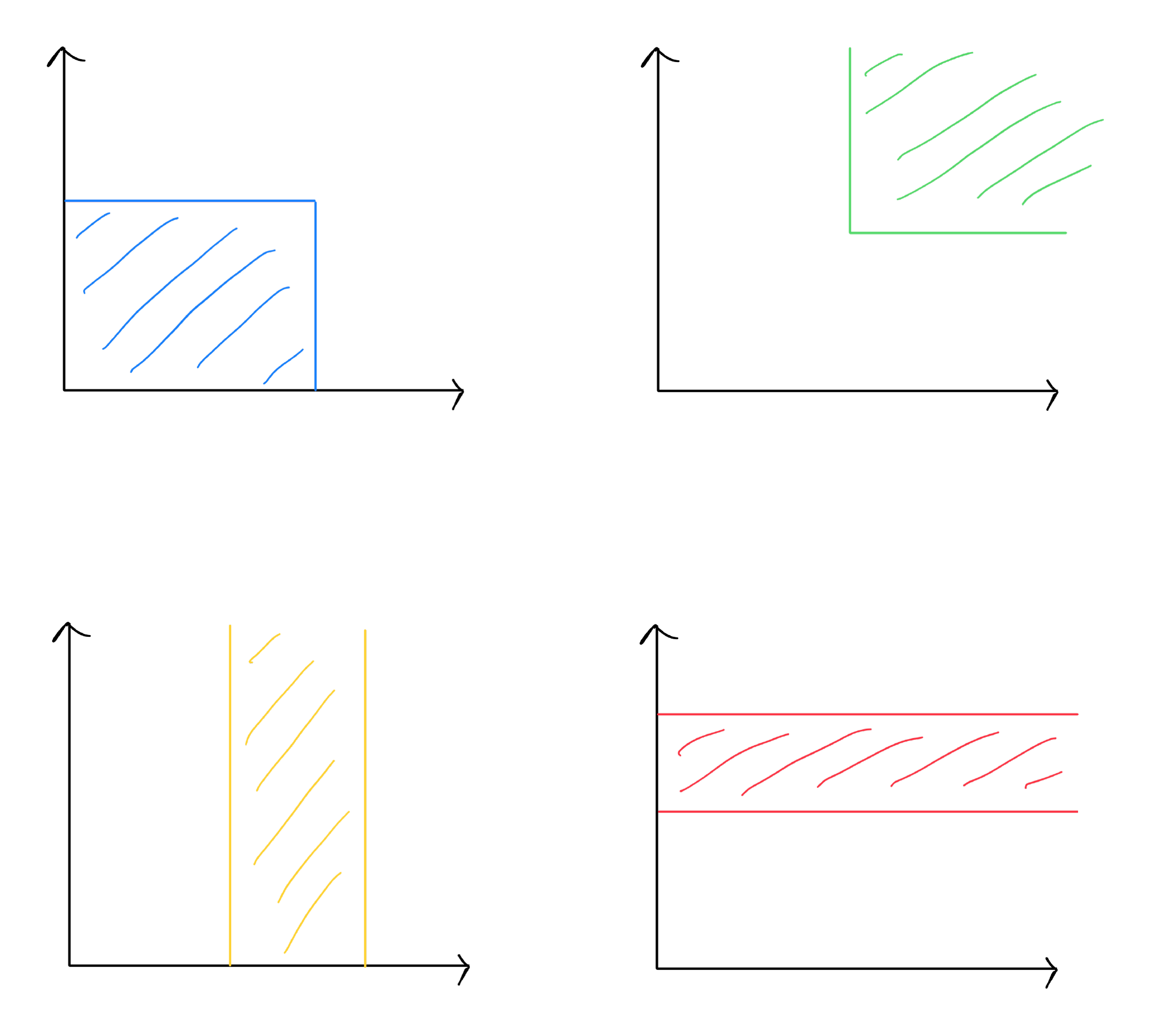}
    \caption{The four types of blocks defined in \autoref{def:blocks}. From the top left: a death block, a birth block, a vertical block, and a horizontal block.}
\end{figure}

\begin{proposition}\label{prop:codegree_1_interval_decomposable}
    A codegree 1 p.f.d. functor $F \colon [0,\infty]^2 \to \Vecs_{\Fb}$ decomposes into a direct sum of interval modules over death blocks, vertical blocks and horizontal blocks.
    
    Dually, a degree 1 p.f.d. functor $F \colon [0,\infty]^2 \to \Vecs_{\Fb}$ decomposes into a direct sum of interval modules over birth blocks, vertical blocks and horizontal blocks.
\end{proposition}

\begin{figure}[htbp]
    \centering
    \includegraphics[width=0.6\linewidth]{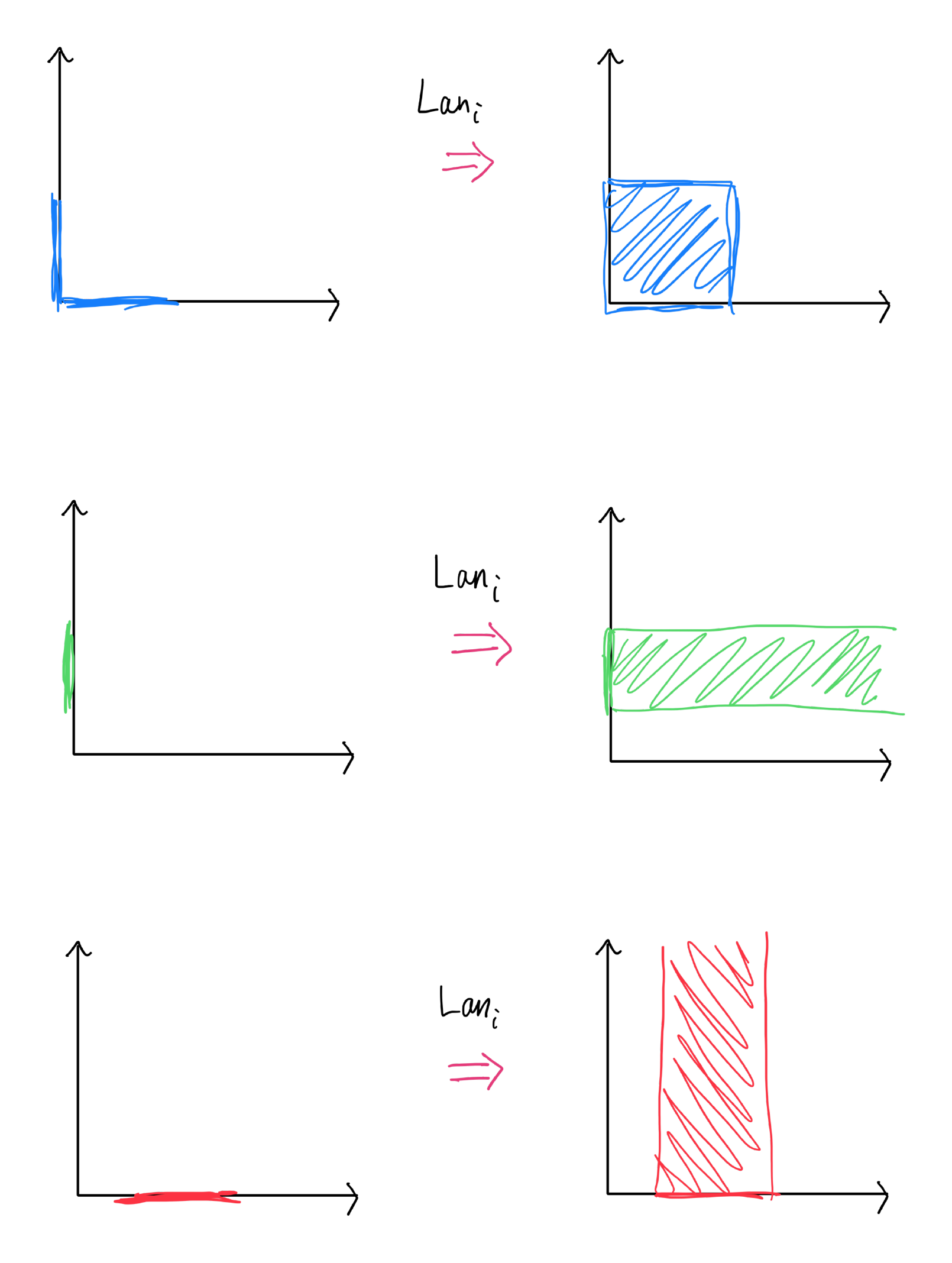}
    \caption{Taking left Kan extensions of different interval modules.}
    \label{fig:blocks_illustration}
\end{figure}

\begin{proof}
    We prove the first part. The second is dual.

    Let $P = [0, \infty]^2$.
    By the theory of poset cocalculus,
    \begin{equation*}
        F \cong T_1 F = \Lan_i F|_{P_{\le 1}},
    \end{equation*}
    where $P_{\le 1} = ([0, \infty] \times \{0\}) \cup (\{0\} \times [0, \infty])$, and $\Lan_i$ denotes the left Kan extension over the inclusion $i \colon P_{\le 1} \hookrightarrow [0,\infty]^2$. By \autoref{lemma:axises_decomposable}, $F|_{P_{\le 1}}$ is interval decomposable.

    We write
    \begin{equation*}
        F|_{P_{\le 1}} = \bigoplus_{I \in \mathcal{B}} \Fb_{I}.
    \end{equation*}
    Now,
    \begin{equation*}
        T_1 F = \Lan_i \bigoplus_{I \in \mathcal{B}} \Fb_{I}
        \ \cong  \bigoplus_{I \in \mathcal{B}} \Lan_i \Fb_{I}
    \end{equation*}
    
    We show that for each $I \in \mathcal{B}$, $\Lan_i \Fb_I$ is an interval module over a death/vertical/horizontal block.
    We can classify the intervals in $\mathcal{B}$ into three types:
    \begin{enumerate}
        \item Intervals containing $(0,0)$.
        \item Intervals lying entirely in the positive part of the $y$-axis.
        \item Intervals lying entirely in the positive part of the $x$-axis.
    \end{enumerate}
    It is straightforward to check that taking $\Lan_i$ over an interval module of one of these types gives an interval module over a death block, horizontal block or vertical block, respectively. This is illustrated in figure \autoref{fig:blocks_illustration}.
    
    We show the computation explicitly for case (1). The other cases can be computed similarly. Let $I$ be an interval in $P_{\le 1}$. Then $I = (J_1 \times \{0\}) \cup (\{0\} \times J_2)$ for some nonempty down-closed subsets $J_1, J_2 \subseteq [0, \infty]$. We now compute $\Lan_i \Fb_I$ objectwise at $(x,y) \in [0,\infty]^2$.
    \begin{equation*}
        \Lan_i \Fb_I (x,y) \cong
        \colim \left(
            \begin{tikzcd}[column sep = scriptsize]
            \Fb_I(0,0) \arrow[d] \arrow[r] & \Fb_I(x,0) \\
            \Fb_I(0,y)                     &  
            \end{tikzcd}
        \right)
    \end{equation*}
    \begin{equation*}
        \cong \begin{cases}
            \Fb, \quad x \in J_1 \textrm{ and } y \in J_2, \\
            0, \quad \textrm{otherwise.}
        \end{cases}
    \end{equation*}
    This is an interval module over the down-closed subset $J_1 \times J_2 \subseteq [0,\infty]^2$.
\end{proof}

\begin{example}
We illustrate how to compute the interval decomposition of a codegree 1 functor.

The following functor, $F \colon \{0,1,2\}^2 \to \Vecs_{\Fb}$, is codegree 1.
\begin{center}
\begin{tikzpicture}
    \diagram{d}{4em}{4em}{
        \Fb & \Fb^2 & \Fb^2 \\
        \Fb & \Fb^2 & \Fb \\
        \Fb & \Fb^2 & \Fb \\
    };
    
    \path[->,font = \scriptsize, midway]
    (d-1-1) edge node[above]{$\begin{pmatrix}1 \\ 0\end{pmatrix}$} (d-1-2)
    (d-1-2) edge node[above]{$1$} (d-1-3)
    (d-2-1) edge node[above]{$\begin{pmatrix}1 \\ 0\end{pmatrix}$} (d-2-2)
    (d-2-2) edge node[above]{$\begin{pmatrix}0 & 1\end{pmatrix}$} (d-2-3)
    (d-3-1) edge node[above]{$\begin{pmatrix}1 \\ 0\end{pmatrix}$} (d-3-2)
    (d-3-2) edge node[above]{$\begin{pmatrix}0 & 1\end{pmatrix}$} (d-3-3)
    (d-2-1) edge node[left]{$0$} (d-1-1)
    (d-2-2) edge node[right]{$\begin{pmatrix}0 & 0 \\ 0 & 1\end{pmatrix}$} (d-1-2)
    (d-2-3) edge node[right]{$\begin{pmatrix}0 \\ 1\end{pmatrix}$}(d-1-3)
    (d-3-1) edge node[left]{$1$} (d-2-1)
    (d-3-2) edge node[right]{$1$} (d-2-2)
    (d-3-3) edge node[right]{$1$} (d-2-3);
\end{tikzpicture}
\end{center}

The restriction $F|_{P_{\le 1}}$ is:
\begin{center}
\begin{tikzpicture}
    \diagram{d}{3em}{3em}{
        \Fb & \ & \ \\
        \Fb & \ & \ \\
        \Fb & \Fb^2 & \Fb \\
    };
    
    \path[->,font = \scriptsize, midway]
    (d-3-1) edge node[above]{$\begin{pmatrix}1 \\ 0\end{pmatrix}$} (d-3-2)
    (d-3-2) edge node[above]{$\begin{pmatrix}0 & 1\end{pmatrix}$} (d-3-3)
    (d-2-1) edge node[left]{$0$} (d-1-1)
    (d-3-1) edge node[left]{$1$} (d-2-1);
\end{tikzpicture}
\end{center}

An interval decomposition of $F|_{P_{\le 1}}$ is $\bigoplus_{I \in \mathcal{B}} \Fb_I$, where
\begin{align*}
    \mathcal{B} = \{& \{(1,0), (2,0)\}, \\
    & \{(0,2)\} \\
    & \{(0,1), (0,0), (1,0)\} \}.
\end{align*}

We compute $\Lan_i \Fb_I$ for each of these intervals.
\begin{center}
\begin{tikzpicture}
    \diagram{d}{2em}{2em}{
        0 & \ & \ & \ & \ & 0 & \Fb & \Fb \\
        0 & \ & \ & \ & \ & 0 & \Fb & \Fb \\
        0 & \Fb & \Fb & \ & \ & 0 & \Fb & \Fb \\
    };
    
    \path[->,font = \scriptsize, midway]
    (d-3-1) edge (d-3-2)
    (d-3-2) edge node[above]{$1$} (d-3-3)
    (d-2-1) edge (d-1-1)
    (d-3-1) edge (d-2-1)
    (d-2-4) edge[thick] node[above]{\large{$\Lan_i$}} (d-2-5)
    (d-1-6) edge (d-1-7)
    (d-1-7) edge node[above]{$1$} (d-1-8)
    (d-2-6) edge (d-2-7)
    (d-2-7) edge node[above]{$1$} (d-2-8)
    (d-3-6) edge (d-3-7)
    (d-3-7) edge node[above]{$1$} (d-3-8)
    (d-2-6) edge (d-1-6)
    (d-2-7) edge node[right]{$1$} (d-1-7)
    (d-2-8) edge node[right]{$1$}(d-1-8)
    (d-3-6) edge (d-2-6)
    (d-3-7) edge node[right]{$1$} (d-2-7)
    (d-3-8) edge node[right]{$1$} (d-2-8);
\end{tikzpicture}
\end{center}

\begin{center}
\begin{tikzpicture}
    \diagram{d}{2em}{2em}{
        \Fb & \ & \ & \ & \ & \Fb & \Fb & \Fb \\
        0 & \ & \ & \ & \ & 0 & 0 & 0 \\
        0 & 0 & 0 & \ & \ & 0 & 0 & 0 \\
    };
    
    \path[->,font = \scriptsize, midway]
    (d-3-1) edge (d-3-2)
    (d-3-2) edge (d-3-3)
    (d-2-1) edge (d-1-1)
    (d-3-1) edge (d-2-1)
    (d-2-4) edge[thick] node[above]{\large{$\Lan_i$}} (d-2-5)
    (d-1-6) edge node[above]{$1$} (d-1-7)
    (d-1-7) edge node[above]{$1$} (d-1-8)
    (d-2-6) edge (d-2-7)
    (d-2-7) edge (d-2-8)
    (d-3-6) edge (d-3-7)
    (d-3-7) edge (d-3-8)
    (d-2-6) edge (d-1-6)
    (d-2-7) edge (d-1-7)
    (d-2-8) edge (d-1-8)
    (d-3-6) edge (d-2-6)
    (d-3-7) edge (d-2-7)
    (d-3-8) edge (d-2-8);
\end{tikzpicture}
\end{center}

\begin{center}
\begin{tikzpicture}
    \diagram{d}{2em}{2em}{
        0 & \ & \ & \ & \ & 0 & 0 & 0 \\
        \Fb & \ & \ & \ & \ & \Fb & \Fb & 0 \\
        \Fb & \Fb & 0 & \ & \ & \Fb & \Fb & 0 \\
    };
    
    \path[->,font = \scriptsize, midway]
    (d-3-1) edge node[above]{$1$} (d-3-2)
    (d-3-2) edge (d-3-3)
    (d-2-1) edge (d-1-1)
    (d-3-1) edge node[left]{$1$} (d-2-1)
    (d-2-4) edge[thick] node[above]{\large{$\Lan_i$}} (d-2-5)
    (d-1-6) edge (d-1-7)
    (d-1-7) edge (d-1-8)
    (d-2-6) edge node[above]{$1$} (d-2-7)
    (d-2-7) edge (d-2-8)
    (d-3-6) edge node[above]{$1$} (d-3-7)
    (d-3-7) edge (d-3-8)
    (d-2-6) edge (d-1-6)
    (d-2-7) edge (d-1-7)
    (d-2-8) edge (d-1-8)
    (d-3-6) edge node[left]{$1$} (d-2-6)
    (d-3-7) edge node[right]{$1$} (d-2-7)
    (d-3-8) edge (d-2-8);
\end{tikzpicture}
\end{center}

The three blocks on the right give the interval decomposition of $F$, i.e., $F$ is isomorphic to the direct sum of these three blocks.
\end{example}

\section{Homotopy (co)degree 1 bipersistence modules}

In this section, we describe the relation between (co)degree 1 functors into $\Vecs_{\Fb}$ and homotopy (co)degree 1 functors into $\Ch_{\Fb}$.

Note that as $\Ch_{\Fb}$ is a stable model category, an $n$-cube is homotopy cartesian if and only if it is homotopy cocartesian \cite[Proposition 1.2.4.13]{lurie2017higher}.
Hence, a functor $\widehat F \colon P \to \Ch_{\Fb}$ is homotopy degree $n$ if and only if it is homotopy codegree $n$.

Let $\iota$ denote the inclusion functor $\iota \colon \Vecs_{\Fb} \to \Ch_{\Fb}$ defined by
\begin{equation*}
    \iota(V)_{\bullet} = \begin{cases}
        V, \quad \bullet = 0,\\
        0, \quad \textrm{otherwise.}
    \end{cases}
\end{equation*}
Observe that for any $F \colon P \to \Vecs_\Fb$, there exists a functor $\iota_* F \colon P \hookrightarrow \Ch_\Fb$.

\begin{lemma}\label{lemma:structure_of_cof_repl}
    Let $P$ be a poset, and $F \colon P \to \Vecs_\Fb$ a functor.
    Then $\iota_* F \colon P \to \Ch_\Fb$ admits:
    \begin{itemize}
        \item a cofibrant replacement $Q$ such that for all $x \in P$ and $i < 0$, $Q(x)_i = 0$, and
        \item a fibrant replacement $R$ such that for all $x \in P$ and $i > 0$, $R(x)_i = 0$.
    \end{itemize}
\end{lemma}

\begin{proof}
    We show the first part. The proof of the second is dual.
    Let $\Ch_{\Fb}^+$ denote the model category of non-negatively graded chain complexes, equipped with the projective model structure (i.e., where the weak equivalences are the quasi-isomorphisms, and the fibrations are the morphisms that are epimorphism in each positive degree). 

    The inclusion $j\colon \Ch_\Fb^+ \hookrightarrow \Ch_\Fb$ has a right adjoint $R$ given by
    \begin{equation*}
        R(C)_i = \begin{cases}
            C_i, &\quad i > 0, \\
            \ker(\partial_0 \colon C_0 \to C_{-1}), &\quad i=0,
        \end{cases}
    \end{equation*}
    which induces a Quillen adjunction
    \begin{equation*}
        \begin{tikzcd}[column sep=large]
            \Fun(P, \Ch_\Fb^+) \arrow[r, shift left=1ex, "j_*"{name=G, yshift=1pt}] & \Fun(P, \Ch_\Fb). \arrow[l, shift left=.5ex, "R_*"{name=F}]
            \arrow[phantom, from=F, to=G, , "\scriptscriptstyle\boldsymbol{\bot}"]
        \end{tikzcd} 
    \end{equation*}
    Now, let $\iota^+$ denote the inclusion $\Vecs_\Fb \hookrightarrow \Ch_\Fb^+$, and note that $\iota = j \circ \iota^+$.
    
    Choose a cofibrant replacement $Q$ of $\iota^+_* F$. Then, $j_* Q$ is a cofibrant replacement of $\iota_*(F)$ in $\Ch_\Fb$, and satisfies $(j_* Q)_i = 0$ for all $i < 0$.
\end{proof}

\begin{proposition}\label{prop:H0_of_T1_tilde}
    Let $P$ be a distributive lattice.
    
    If $F \colon P \to \Vecs_{\Fb}$ is a functor, then 
    \begin{equation*}
        H_0 \widetilde T_n (\iota_* F) \cong T_n F,
    \end{equation*}
    and
    \begin{equation*}
        H_0 \widetilde T^n (\iota_* F) \cong T^n F.
    \end{equation*}

%
\end{proposition}
\begin{proof}
    We prove the first statement. The proof of the second is dual.

    Using \autoref{lemma:structure_of_cof_repl}, let $Q$ be a cofibrant resolution of $(\iota_* F)|_{P_{\le n}}$ such that for all $x \in P$ and $i<0$, $Q(x)_i = 0$.
    Then, for all $x \in P$ and $i<0$,
    \begin{equation*}
        (\widetilde T_n (\iota_* F))(x)_i = \underset{y \le x, y \in P_{\le n}}{\colim} Q(y)_i \cong 0.
    \end{equation*}
    
    Hence,
    \begin{align*}
        (H_0 \widetilde T_n F) (x) &\cong \coker\left(\partial_1 \colon \underset{y \le x, y \in P_{\le n}}{\colim} Q(y)_1 \to \underset{y \le x, y \in P_{\le n}}{\colim} Q(y)_0\right) \\
        &\cong \underset{y \le x, y \in P_{\le n}}{\colim} \coker \left(\partial_1 : Q(y)_1 \to Q(y)_0\right) \\
        &\cong \underset{y \le x, y \in P_{\le n}}{\colim} H_0 Q (y) \cong \underset{y \le x, y \in P_{\le n}}{\colim} F (y) \cong (T_n F)(x),
    \end{align*}
    where we use that colimits commutes with colimits (and hence cokernels).
    This concludes the proof.
    
\end{proof}

For the following corollary, recall that a functor into $\Ch_{\Fb}$ is homotopy degree $n$ if and only if it is homotopy codegree $n$.

\begin{corollary}\label{prop:lifting_deg_1_to_hodeg_1}
    If $P$ is a join-factorization lattice and $F \colon P \to \Vecs_{\Fb}$ a codegree $n$ functor, then there exists a homotopy degree $n$ functor $\widehat F \colon P \to \Ch_{\Fb}$ such that $H_1 \widehat F \cong F$.
    
    Dually, if $P$ is a meet-factorization lattice and $F \colon P \to \Vecs_{\Fb}$ a degree $n$ functor, then there exists a homotopy degree $n$ functor $\widehat F \colon P \to \Ch_{\Fb}$ such that $H_0 \widehat F \cong F$.
\end{corollary}

\begin{proof}
    We show the first case. The second is dual (as a functor in $\Fun(P, \Ch_\Fb)$ is homotopy degree $n$ if and only if it is homotopy codegree $n$).

    Let $\widehat F = \widetilde T_1 (\iota_* F)$. Then, $\widehat F$ is homotopy degree 1, and $H_0 \circ \widehat F \cong T_1 F$ by \autoref{prop:H0_of_T1_tilde}. Furthermore, by \autoref{theoremB_cat}, $T_1 F \cong F$, which concludes the proof.
\end{proof}

\section{Middle exactness and poset cocalculus}

In this section, we study middle exactness in the 2-parameter setting. We establish a connection between middle-exactness and poset cocalculus. As a corollary, we get a new proof of the fact that middle-exact bipersistence modules are block decomposable.

Given two morphisms $f_1 \colon X \to Y$ and $f_2 \colon X \to Z$ in a category, we let $Y \coprod_X Z$ denote the corresponding pushout (if it exists). Similarly, given two morphisms $f_1 \colon Y \to X$ and $f_2 \colon Z \to X$, we let $Y \times_X Z$ denote the corresponding product (if it exists).

\subsection{Properties of middle-exact squares}
\begin{lemma}\label{lemma:me_square_equiv_cond}
    Let $\A$ be an abelian category, and let
    \begin{equation}\label{eq:me_square_cond_lemma}
    \begin{tikzcd}
    A \arrow[r, "\alpha"] \arrow[d, "f"] & B \arrow[d, "g"] \\
    C \arrow[r, "\beta"]           & D,
    \end{tikzcd}
    \end{equation}
    be a square in $\A$.
    The following are equivalent.
    \begin{enumerate}
        \item \eqref{eq:me_square_cond_lemma} is middle-exact.
        \item The canonical map $B \coprod_A C \to D$ is a monomorphism.
        \item The canonical map $A \to B \times_D C$ is an epimorphism.
    \end{enumerate}
\end{lemma}
\begin{proof}
    Consider
    \begin{equation}\label{eq:cplx_square_cond_lemma}
        A \xrightarrow{\begin{pmatrix}
            \alpha \\ f
        \end{pmatrix}} B \oplus C \xrightarrow{\begin{pmatrix}
            g & -\beta
        \end{pmatrix}} D,
    \end{equation}
    and observe that
    \begin{equation*}
        B \coprod_A C \cong \coker\begin{pmatrix}
            \alpha \\ f
        \end{pmatrix},
    \end{equation*}
    and
    \begin{equation*}
        B \times_D C \cong \ker\begin{pmatrix}
            g & -\beta
        \end{pmatrix}.
    \end{equation*}
    By definition, the homology of \eqref{eq:cplx_square_cond_lemma} is defined as
    \begin{equation*}
        \ker {\begin{pmatrix}g & -\beta\end{pmatrix}} \ / \ \Ima{\begin{pmatrix}\alpha \\ f\end{pmatrix}},
    \end{equation*}
    which is isomorphic to
    \begin{equation*}
        \coker \left( A \xrightarrow{\begin{pmatrix}\alpha \\ f\end{pmatrix}} (\ker \begin{pmatrix}g & -\beta\end{pmatrix}) \right),
    \end{equation*}
    and to
    \begin{equation*}
        \ker \left((\coker\begin{pmatrix}\alpha \\ f\end{pmatrix}) \xrightarrow{\begin{pmatrix}g & -\beta\end{pmatrix}} D \right).
    \end{equation*}
    Hence, \eqref{eq:me_square_cond_lemma} is middle-exact if and only if the map $A \to \ker \begin{pmatrix}g & -\beta\end{pmatrix}$ is an epimorphism, and also if and only if the map $\coker\begin{pmatrix}\alpha \\ f\end{pmatrix} \to D$ is a monomorphism.
\end{proof}

\begin{lemma}\label{lemma:me_square_coker_injective}
    Let $\A$ be an abelian category. If
    \begin{center}
    \begin{tikzcd}
    A \arrow[r, "\alpha"] \arrow[d, "f"] & B \arrow[d, "g"] \\
    C \arrow[r, "\beta"]           & D,
    \end{tikzcd}
    \end{center}
    is a middle-exact square in $\A$, then the induced map
    \begin{equation*}
        \bar \beta \colon \coker f \to \coker g
    \end{equation*}
    is a monomorphism, and the induced map
    \begin{equation*}
        \bar \alpha \colon \ker f \to \ker g
    \end{equation*}
    is an epimorphism.
\end{lemma}

\begin{proof}
    Consider the short exact sequence of chain complexes
    \begin{center}
        \begin{tikzcd}[row sep=4em, column sep = 4em]
            0 \arrow[d] \arrow[r]           & A \arrow[r, "="] \arrow[d] \arrow[d, "{\begin{pmatrix}\alpha \\ f\end{pmatrix}}"] & A \arrow[d, "f"] \\
            B \arrow[d, "g"] \arrow[r, "{\begin{pmatrix}1 \\ 0\end{pmatrix}}"] & B \oplus C \arrow[r, "(0 \ 1)"] \arrow[d, "(g \ -\beta)"]       & C \arrow[d]           \\
            D \arrow[r, "="]                & D \arrow[r]                               & 0                    
        \end{tikzcd}
    \end{center}
    By the lemma assumption, the middle column is middle-exact. Hence, from the long exact sequence in homology, we get that the following sequences are short exact:
    \begin{equation*}
        0 \to \ker{\begin{pmatrix}\alpha \\ f\end{pmatrix}} \to \ker f \to \ker g \to 0,
    \end{equation*}
    \begin{equation*}
        0 \to \coker f \to \coker g \to \coker \begin{pmatrix}g & -\beta\end{pmatrix} \to 0.
    \end{equation*}
    The result follows immediately.
\end{proof}

\begin{lemma}\label{lemma:me_and_mono_is_pb}
    Let $\A$ be an abelian category, and let
    \begin{center}
    \begin{tikzcd}
    A \arrow[r, "\alpha"] \arrow[d, "f"] & B \arrow[d, "g"] \\
    C \arrow[r, "\beta"]           & D,
    \end{tikzcd}
    \end{center}
    be a middle-exact square in $\A$.

    If either $\alpha$ or $f$ is a monomorphism, then the square is a pullback.

    Dually, if either $\beta$ or $g$ is an epimorphism, then the square is a pushout.
\end{lemma}

\begin{proof}
    We prove the first part. The second is dual.
    
    Assume without loss of generality the $\alpha$ is a monomorphism.
    As the square is middle-exact, the associated complex
    \begin{equation*}
        A \xrightarrow{\begin{pmatrix}
            \alpha \\ f
        \end{pmatrix}} B \oplus C \xrightarrow{\begin{pmatrix}
            g & -\beta
        \end{pmatrix}} D,
    \end{equation*}
    is exact in the middle. Furthermore, as $\alpha$ is a monomorphism, so is the map $A \to B \oplus C$.

    Hence, the complex is left exact. Thus,
    \begin{equation*}
        A \cong \ker \left(B \oplus C \xrightarrow{\begin{pmatrix}
            g & -\beta
        \end{pmatrix}} D,\right)
        \cong B \times_D C,
    \end{equation*}
    as desired.
\end{proof}

\subsection{Characterization of middle exactness using poset (co)calculus}

\begin{lemma}\label{lemma:co_limits_preserve_middle_exactness}
    Let $G, F \colon [0, \infty]^2 \to \Vecs_\Fb$ be functors and $\eta \colon G \to F$ a natural transformation. 
    If $G$ is codegree 1, and $F$ is middle-exact, then the functor
    \begin{equation*}
        \coker(G \xrightarrow{\eta} F)
    \end{equation*}
    is middle-exact.

    Dually, if $G$ is middle-exact, and $F$ is degree 1, then the functor
    \begin{equation*}
        \ker(G \xrightarrow{\eta} F)
    \end{equation*}
    is middle-exact.
\end{lemma}
\begin{proof}
    We prove the first part. The proof of the second part is dual.
    
    Let $F$ and $G$ be as in the lemma statement.
    By \autoref{lemma:me_square_equiv_cond} and \autoref{def:middle_exactness}, a functor $H \colon [0, \infty]^2 \to \Vecs_\Fb$ is middle-exact if and only if for all $p,q \in [0, \infty]^2$, the map
    \begin{equation*}
        \colim \left(
            \begin{tikzcd}[column sep = scriptsize]
            H(p \wedge q) \arrow[d] \arrow[r] & H(p) \\
            H(q)                     &  
            \end{tikzcd}
        \right) \to H(p \vee q)
    \end{equation*}
    is a split monomorphism.
    
    We have that
    \begin{equation*}
        \colim \left(
            \begin{tikzcd}[column sep = scriptsize]
            \coker(G \to F)(p \wedge q) \arrow[d] \arrow[r] & \coker(G \to F)(p) \\
            \coker(G \to F)(q)                     &  
            \end{tikzcd}
        \right)
    \end{equation*}
    \begin{equation*}
        \cong \quad
        \coker \left( \colim \left(
            \begin{tikzcd}
            G(p \wedge q) \arrow[d] \arrow[r] & G(p) \\
            G(q)                     &  
            \end{tikzcd}
        \right)
        \to
        \colim \left(
            \begin{tikzcd}[column sep = scriptsize]
            F(p \wedge q) \arrow[d] \arrow[r] & F(p) \\
            F(q)                     &  
            \end{tikzcd}
        \right)
        \right)
    \end{equation*}
    \begin{equation*}
        \cong \quad
        \coker \left(G(p \vee q) \to
        \colim \left(
            \begin{tikzcd}
            F(p \wedge q) \arrow[d] \arrow[r] & F(p) \\
            F(q)                     &  
            \end{tikzcd}
        \right)
        \right)
    \end{equation*}
    as $G$ is codegree 1.
    
    Now, as $F$ is middle-exact, the map
    \begin{equation*}
        \colim \left(
            \begin{tikzcd}[column sep = scriptsize]
            F(p \wedge q) \arrow[d] \arrow[r] & F(p) \\
            F(q)                     &  
            \end{tikzcd}
        \right)
        \to F(p \vee q)
    \end{equation*}
    is a monomorphism, and thus the map
    \begin{equation*}
        \coker \left(G(p \vee q) \to
        \colim \left(
            \begin{tikzcd}[column sep = scriptsize]
            F(p \wedge q) \arrow[d] \arrow[r] & F(p) \\
            F(q)                     &  
            \end{tikzcd}
        \right)
        \right)
        \longrightarrow \coker \left(G(p \vee q) \to F(p \vee q)\right)
    \end{equation*}
    is also a monomorphism.
\end{proof}

\begin{definition}\label{def:increasing_functor}
    Let $J$ be a small category. A functor $F \colon J \to \Vecs_\Fb$ is \emph{monomorphic} if $F(f)$ is a monomorphism for every morphism $f$ in $J$.
    Similarly, $F$ is \emph{epimorphic} if $F(f)$ is an epimorphism for every morphism $f$ in $J$.
\end{definition}

\begin{lemma}\label{lemma:me_increasing_is_deg1}
    Let $F \colon [0, \infty]^2 \to \Vecs_\Fb$ be middle-exact.
    If $F$ is monomorphic, then $F$ is degree 1.

    Dually, if $F$ is epimorphic, then $F$ is codegree 1.
\end{lemma}

\begin{proof}
    This follows directly from \autoref{lemma:me_and_mono_is_pb}.
%
\end{proof}

\begin{lemma}\label{lemma:me_deg1_coker_increasing}
    Let $F \colon [0, \infty]^2 \to \Vecs_\Fb$ be middle-exact. Then the functor
    \begin{equation*}
        C = \coker \left( T_1 F \to F \right) \colon [0, \infty^2] \to \Vecs_{\Fb}
    \end{equation*}
    is monomorphic and degree 1.

    Dually, the functor
    \begin{equation*}
        K = \ker \left( F \to T^1 F \right) \colon [0, \infty^2] \to \Vecs_{\Fb}
    \end{equation*}
    is epimorphic and codegree 1.
\end{lemma}

\begin{proof}
    We show the first part. The proof of the second part is dual.

    We first show that $C$ is monomorphic. We will show that for all $(x,y) \in [0,\infty]^2$ and $x' \ge x$, the map $C((x,y) \le (x',y))$ is a monomorphism. One can similarly show that maps of the form $C(x, y \le y')$ are monomorphisms, which will imply that for all $(x,y) \le (x',y')$, the map $C((x,y) \le (x',y')) = C(x,y\le y') \circ C(x\le x', y)$ is a monomorphism.
    
    Observe that it follows from \autoref{lemma:co_limits_preserve_middle_exactness} that $C$ is middle-exact.
    Let $x,y$ and $x'$ be as above, and
    consider the square
    \begin{center}
    \begin{tikzcd}
    C(x, 0) \arrow[d] \arrow[r] & C(x,y) \arrow[d] \\
    C(x', 0) \arrow[r]          & C(x',y).
    \end{tikzcd}
    \end{center}
    As $(x,0)$ and $(x',0)$ are in $P_{\le 1}$, it follows that $C(x,0) \cong C(x',0) \cong 0$. Thus, by middle-exactness, $C((x,y)\le(x',y))$ is a monomorphism.
    
    Hence, by \autoref{lemma:me_increasing_is_deg1}, $C$ is degree 1, which concludes the proof.

%
%
\end{proof}


\begin{lemma}\label{lemma:increasing_deg_1_is_free}
    A codegree 1, p.f.d., epimorphic functor $F \colon [0, \infty]^2 \to \Vecs_\Fb$ is an injective object in $\Fun([0, \infty^2], \Vecs_{\Fb})$.
\end{lemma}

\begin{proof}
    Let $P = [0, \infty]^2$.
    Recall that $T_1 F$ is defined as the left Kan extension
    \begin{equation*}
        T_1 F = \underset{i}{\Lan} \ F|_{P_{\le 1}},
    \end{equation*}
    where $i$ is the inclusion functor $i \colon P_{\le 1} \hookrightarrow P$.
    Note further that $P_{\le 1} = ([0, \infty] \times \{0\}) \cup (\{0\} \times [0, \infty])$.
    By \autoref{lemma:axises_decomposable}, $F|_{P_{\le 1}}$ is interval decomposable. We write
    \begin{equation*}
        F|_{P_{\le 1}} = \bigoplus_{I \in \mathcal{B}} \Fb_{I}
    \end{equation*}
    As $F$ is epimorphic, each $I \in \mathcal{B}$ is of the form $I = (I_1 \times \{0\}) \cup (\{0\} \times I_2)$, where each $I_i$ is of the form $[0, u_i]$ or $[0, u_i)$ for some $u_i \in [0, \infty]$. Moreover, as $F$ is p.f.d., in particular $F(0)$ is finite dimensional, and hence $|\B|$ is finite. It now suffices to show that $\Lan_{i} (\Fb_I)$ is injective, where $I$ is an interval of this form, as $T_1$ commutes with direct sums, and a finite direct sum of injective modules is injective. 

    A straightforward computation shows that
    \begin{equation*}
        \underset{i}{\Lan} \ \Fb|_{(I_1 \times \{0\}) \cup (\{0\} \times I_2)}
        = \Fb_{I_1 \times I_2},
    \end{equation*}
    which is injective by \autoref{lemma:directed_ideal_injective}.
\end{proof}

\begin{theorem}\label{thm:main_theorem}
    Let $F \colon [0, \infty]^2 \to \Vecs_{\Fb}$ be a p.f.d. functor. The following are equivalent.
    \begin{description}
        \item[(i)] $F$ is middle-exact.
        \item[(ii)] $F$ is the direct sum of a degree 1 functor and a codegree 1 functor.
        \item[(iii)] There exists a homotopy degree 1 functor $\widehat F \colon [0, \infty]^2 \to \Ch_{\Fb}$ such that $F \cong H_0 \circ \widehat F$.
    \end{description}
\end{theorem}

\begin{proof}
\begin{description}
    \item[(i) $\implies$ (ii)] Let $F \colon [0, \infty]^2 \to \Vecs_{\Fb}$ be middle-exact.

    Let $K = \ker(T_1 F \to F)$.
    Observe that for each $(x,y) \in [0, \infty]^2$, $F(x, y) \to T^1 F(x,y)$ is the map
    \begin{equation*}
        F(x,y) \to \lim \left(
            \begin{tikzcd}
            & F(x,\infty) \arrow[d] \\
            F(\infty,y) \arrow[r] & F(\infty,\infty)  
            \end{tikzcd}
        \right),
    \end{equation*}
    which is an epimorphism by \autoref{lemma:me_square_equiv_cond}.
    We therefore get a short exact sequence in the abelian category $\Fun([0, \infty]^2, \Vecs_\Fb)$,
    \begin{equation*}
        0 \to K \to F \to T^1 F \to 0.
    \end{equation*}
    By \autoref{lemma:me_deg1_coker_increasing}, $K$ is codegree 1 and epimorphic. Moreover, $K$ is p.f.d. as $F$ is.
    Furthermore, by \autoref{lemma:increasing_deg_1_is_free}, $K$ is an injective, so the short exact sequence splits.
    In total, we get that
    \begin{equation*}
        F \cong T^1 F \oplus K,
    \end{equation*}
    as desired.

    \item[(ii) $\implies$ (iii)]
    This follows directly from \autoref{prop:lifting_deg_1_to_hodeg_1} and the fact that $H_0$ is an additive functor.
    
    \item[(iii) $\implies$ (i)]
    Let $\widehat F \colon [0, \infty]^2 \to \Ch_{\Fb}$ be homotopy degree 1.
    Let $p, q \in [0, \infty]^2$. As the square
    \begin{center}
    \begin{tikzcd}
    \widehat F(p \wedge q) \arrow[r] \arrow[d] & \widehat F(p) \arrow[d] \\
    \widehat F(q) \arrow[r] & \widehat F(p \vee q)
    \end{tikzcd}
    \end{center}
    is a homotopy pushout, we get from the long exact sequence in homology that
    \begin{equation*}
        H_0(\widehat F(p \wedge q)) \to H_0(\widehat F(p)) \oplus H_0(\widehat F(q)) \to H_0(\widehat F(p \vee q))
    \end{equation*}
    is middle-exact.
\end{description}

\end{proof}

We now show how \autoref{thm:main_theorem} leads to an alternative proof of the fact that middle-exact persistence modules over $[0, \infty]^2$ are interval decomposable \cite[Theorem 1.3]{botnanMiddleExactness}. Recall the definition of blocks from \autoref{def:blocks}.

\begin{corollary}
    Let $F \colon [0, \infty]^2 \to \Vecs_{\Fb}$ be a p.f.d. middle-exact functor. Then $F$ decomposes into a direct sum of interval modules over blocks.
\end{corollary}

\begin{proof}
    By \autoref{thm:main_theorem}, $F$ is a direct sum of a codegree 1 functor and a degree 1 functor. The result now follows from \autoref{prop:codegree_1_interval_decomposable}.
\end{proof}

\section{Relation to the Reedy model structure}

In this section, we apply the \emph{Reedy model structure} on functor categories to establish a characterization of projective and injective objects in $\Fun(P, \Vecs_{\Fb})$ for a poset $P$ (under certain conditions on the poset).
We then give an example illustrating how this model structure is connected to middle-exact bipersistence modules.

\subsection{Notation and preliminaries on posets}

For two elements $x,y$ in a poset $P$, we say that $y$ \emph{covers} $x$ if $y > x$ and that there exists no $z$ with $y > z > x$. We denote this by $y \succ x$. If $y$ covers $x$, we say that $x$ is a \emph{parent} of $y$.

%
%

\begin{lemma}\label{lemma:join_of_parents}
    Let $P$ be a lattice, and let $u \in P$. If $a,b \prec u$ with $a \neq b$, then $a \vee b = u$.
\end{lemma}
\begin{proof}
    As $a \le a \vee b \le u$, either $a \vee b = a$ or $a \vee b = u$. If $a \vee b = a$, then $b \le a < u$, so either $a = b$ or $b < a < u$, both of which contradict the assumptions.
\end{proof}
Remark that \autoref{lemma:join_of_parents} in particular implies that join-irreducible elements have at most one parent.

For an element $a \in P$, let $a^\downarrow$ denote the interval $a^\downarrow = \{x \in P : x \le a\}$.
Let $a^\uparrow$ denote the interval $a^\uparrow = \{x \in P : x \ge a\}$. 
\begin{definition}
    A poset $P$ is said to be \emph{down-finite} if, for every $a \in P$, $a^\downarrow$ is finite.
    
    A poset $P$ is said to be \emph{up-finite} if, for every $a \in P$, $a^\uparrow$ is finite.
\end{definition}
Observe that down-finite posets satisfy the descending chain condition, and up-finite posets satisfy the ascending chain condition.

\begin{example}
    Any finite poset is down-finite. The poset $(\Nn, \le)$ is down-finite, but not finite. The poset $(\Nn \cup \{\infty\}, \le)$, defined by adding a greatest element $\infty$ to $(\Nn, \le)$, satisfies the descending chain condition, but is not down-finite.
\end{example}

\begin{lemma}\label{lemma:num_of_parents}
    Let $P$ be a down-finite distributive lattice, and let $x \in P$. Then $x$ has exactly $\jdim(x)$ parents.
\end{lemma}
\begin{proof}
    This follows from \cite[Lemma 4.20]{hem2025posetfunctorcocalculusapplications} and \cite[Proposition 5.8]{realisationsposetstameness}.
\end{proof}

\subsection{The Reedy model structure and projective objects}

Let $P$ be a finite poset.
We can equip $\Vecs_{\Fb}$ with the model structure where the cofibrations are monomorphisms, fibrations are epimorphisms and weak equivalences are all morphisms.
As $P$ satisfies the descending chain condition, we have a Reedy model structure on $\Fun(P, \Vecs_{\Fb})$ that coincides with the projective model structure (see, e.g., \cite[Chapter 5.2]{hovey}), which we denote $\projPvec$. In $\projPvec$, all morphisms are weak equivalences, the fibrations are the objectwise epimorphisms, and the cofibrations are the natural transformations $\eta \colon F \to G$ such that for all $x \in P$, the \emph{latching map}
\begin{equation}\label{eq:latching_map}
    L_{x} G \coprod_{L_{x} F} F(x) \to G(x)
\end{equation}
is a monomorphism, where $L_{x} F$ is the latching object of $F$ at $x$ and is given by
\begin{equation}\label{eq:latching_ob}
    L_{x} F = \underset{y < x}{\colim} F(y).
\end{equation}
Observe that the cofibrant objects are precisely the functors that have the left lifting property against objectwise epimorphisms. In other words, the cofibrant objects are the projective objects in $\Fun(P, \Vecs_{\Fb})$.
Note further that $F \in \projPvec$ is cofibrant if and only if
\begin{equation}\label{eq:latching_map_cofibrant}
    L_{x} F \to F(x)
\end{equation}
is a monomorphism for all $x \in P$.

A \emph{meet-sublattice} $Q$ of a lattice $P$ is a full subposet $Q \subset P$ such that for any $x, y \in Q$, $x \wedge_P y \in Q$, where $\wedge_p$ is the meet operation in $P$.
\begin{lemma}\label{lemma:restricting_cofibrant}
    Let $P$ be a lattice satisfying the descending chain condition, and $Q \subset P$ a finite meet-sublattice.
    If $F$ is a projective object in $\Fun(P, \Vecs_{\Fb})$, then the restriction $F|_Q$ is a projective object in $\Fun(Q, \Vecs_\Fb)$.
\end{lemma}
\begin{proof}
    Let $m$ be the maximal element in $Q$ and let $P' = \{x \in P : x \le m\}$, viewed as a subposet of $P$. By inspecting equations \eqref{eq:latching_ob} and \eqref{eq:latching_map_cofibrant}, we see that $F|_{P'}$ is a cofibrant object in $\Fun_{\textrm{proj}}(P', \Vecs_\Fb)$.

    Let $i$ denote the inclusion $i : Q \hookrightarrow P'$. Let
    \begin{align*}
        r \colon &P' \to Q \\
        &x \mapsto \bigwedge_{u \in Q, u \ge x} u,
    \end{align*}
    which is well-defined as $Q$ is a sublattice.
    Then $ri(u) = u$ for all $u \in Q$, and $ir(x) \ge x$ for all $x \in P'$ (as $Q$ contains the maximal element in $P'$).
    Hence, $r \dashv i$ forms an adjoint pair, and thus $i^* \colon \Fun(P', \Vecs_{\Fb}) \to \Fun(Q, \Vecs_\Fb)$ is left Quillen adjoint to $r^*$. Thus, $F|_Q = F|_{P'} \circ i$ is cofibrant, and hence a projective object in $\Fun(Q, \Vecs_\Fb)$.
\end{proof}

\begin{lemma}\label{lemma:projective_is_deg1}
    Let $P$ be a lattice satisfying the descending chain condition.
    If $F$ is a projective object in $\Fun(P, \Vecs_\Fb)$, then $F$ is monomorphic and degree 1.
\end{lemma}

\begin{proof}
    We first prove that $F$ is monomorphic. Let $x, y \in P$ with $y \le x$. Let $Q = \{x,y\} \subset P$. By \autoref{lemma:restricting_cofibrant}, $F|_Q$ is a projective object in $\Fun(Q, \Vecs_{\Fb})$. Applying \eqref{eq:latching_map_cofibrant}, we get that
    \begin{equation*}
        F(y) \to F(x)
    \end{equation*}
    is a monomorphism. Hence, $F$ is monomorphic.

    Now, let $x, y \in P$, and consider the sublattice $Q = \{x,y,x\wedge y, x \vee y\}$. By \autoref{lemma:restricting_cofibrant}, $F|_Q$ is a projective object in $\Fun(Q, \Vecs_{\Fb})$. Applying \eqref{eq:latching_map_cofibrant}, we get that
    \begin{equation*}
        \colim \left(
            \begin{tikzcd}[column sep = scriptsize]
            F(x \wedge y) \arrow[d] \arrow[r] & F(x) \\
            F(y)                     &  
            \end{tikzcd}
        \right) \to F(x \vee y)
    \end{equation*}
    is a monomorphism.
    Thus, the square
    \begin{equation*}
        \begin{tikzcd}[column sep = scriptsize]
            F(x \wedge y) \arrow[d] \arrow[r] & F(x) \arrow[d] \\
            F(y)   \arrow[r]     &  F (x \vee y)
        \end{tikzcd}
    \end{equation*}
    is middle-exact, by \autoref{lemma:me_square_equiv_cond}.
    Furthermore, as $F$ is monomorphic, the square is a pullback, by \autoref{lemma:me_and_mono_is_pb}.
    Thus, $F$ is degree 1.
\end{proof}

\begin{lemma}\label{lemma:latching_ob_cube_dim1}
    Let $P$ be a down-finite distributive lattice, $\C$ a cocomplete category and $F \colon P \to \C$ a functor.
    Let $x \in P$ with $\jdim(x) = 1$. Then,
    \begin{equation*}
        L_x F \cong F(y),
    \end{equation*}
    where $y$ is the unique parent of $x$.
\end{lemma}
\begin{proof}
    Note that $x$ has exactly one parent, by \autoref{lemma:num_of_parents}. Let $y \prec x$ be the parent.
    Then $y$ is maximal in $\{u \in P : u < x\}$. Thus,
    \begin{equation*}
        L_x F \cong \underset{u < x}{\colim} F(u) \cong F(y).
    \end{equation*}
\end{proof}

\begin{lemma}\label{lemma:latching_ob_cube}
    Let $P$ be a down-finite distributive lattice, $\C$ a cocomplete category and $F \colon P \to \C$ a functor.
    Let $x \in P$ with $\jdim(x) = n \ge 2$, and let $u_0, \dots, u_{n-1}$ be the parents of $x$. Then,
    \begin{equation*}
        L_x F \cong \underset{S \subsetneq [n-1]}{\colim} F \circ \X_{u_0, \dots, u_{n-1}}.
    \end{equation*}
\end{lemma}
\begin{proof}
    Recall that by \autoref{lemma:num_of_parents}, the number of parents of $x$ does indeed equal $n$.
    By \autoref{lemma:join_of_parents}, $u_0, \dots, u_{n-1}$ form a pairwise cover of $x$, so by \autoref{lem:bicartesian_from_codecomp}, $\X_{u_0, \dots, u_{n-1}}$ is a well-defined strongly bicartesian $n$-cube.
    
    It suffices to show that the inclusion $\{\X_{u_0, \dots, u_{n-1}}(S) : S \subsetneq [n-1]\} \subseteq \{y \in P : y < x\}$ is a final functor. As the left hand side contains all maximal elements in the right hand side, and is meet-closed, this follows from \cite[Lemma 4.24]{hem2025posetfunctorcocalculusapplications}.
\end{proof}

\begin{lemma}\label{lemma:codeg1_mono_is_proj}
    Let $P$ be a down-finite distributive lattice. If $F \colon P \to \Vecs_\Fb$ is monomorphic and codegree 1, then $F$ is a projective object.
\end{lemma}
\begin{proof}
    We show that $F$ is cofibrant in $\projPvec$.
    We need to verify \autoref{eq:latching_map_cofibrant} for every $x \in P$.
    The equation is trivially satisfied for the minimal element in $P$.
    
    If $\jdim(x) = 1$, then $L_x F = F(y)$, where $y \prec x$, by \autoref{lemma:latching_ob_cube_dim1}.
    As $F$ is monomorphic, the map $F(y) \to F(x)$ is a monomorphism.

    Now suppose $\jdim(x) \ge 2$, and let $u_0, \dots, u_{n-1}$ be the parents of $x$. Then, by \autoref{lemma:latching_ob_cube}, the latching map \eqref{eq:latching_map_cofibrant} is
    \begin{equation*}
        \underset{S \subsetneq [n-1]}{\colim} F \circ \X_{u_0, \dots, u_{n-1}} \ \to \ F(x).
    \end{equation*}
    As $F$ is codegree 1, this is an isomorphism, and in particular a monomorphism. This concludes the proof.
\end{proof}

\begin{example}
We show that every middle-exact functor $F \colon \Nn \times \Nn \to \Vecs_{\Fb}$ decomposes as a direct sum of a degree 1 functor and a codegree 1 functor.
Observe that $\Nn \times \Nn$ is a down-finite distributive lattice, and hence a join-factorization lattice.

We show first that $\varepsilon_1 \colon T_1 F \to F$ is a cofibration in $\projPvec$. We need to verify that the latching map \eqref{eq:latching_map} is a monomorphism for all $x = (u,v) \in P$.

First, suppose that $x = (0,0)$. Then the latching map is
\begin{equation*}
    T_1 F(0,0) \to F(0,0),
\end{equation*}
which is an isomorphism.

Next, suppose that $x$ has dim 1. Assume without loss of generality that $x = (u, 0)$ for some $u > 0$. Then the latching map is
\begin{equation*}
    F(u-1,0) \coprod_{T_1 F(u-1,0)} T_1 F(u,0) \to F(u,0),
\end{equation*}
which is an isomorphism as $T_1 F (u-1,0) = F(u-1,0)$ and $T_1 F (u,0) = F(u,0)$.

Now, suppose that $x = (u,v)$ for some $u > 0, v > 0$. Then the latching map is
\begin{equation}\label{eq:example_reedy}
    L_{(u,v)} F \coprod_{L_{(u,v)} T_1 F} T_1 F(u,v) \to F(u,v).
\end{equation}

As $T_1 F$ is codegree 1, we have
\begin{equation}
   L_{(u,v)} T_1 F \cong \colim \left(
            \begin{tikzcd}[column sep = scriptsize]
            T_1 F(u-1,v-1) \arrow[d] \arrow[r] & T_1 F(u-1, v) \\
            T_1 F(u, v-1)                     &  
            \end{tikzcd}
    \right) 
    \cong T_1 F (u,v).
\end{equation}
Hence, \eqref{eq:example_reedy} simplifies to
\begin{equation*}
    L_{(u,v)} F \to F,
\end{equation*}
which is
\begin{equation}
   \colim \left(
            \begin{tikzcd}[column sep = scriptsize]
            F(u-1,v-1) \arrow[d] \arrow[r] & F(u-1, v) \\
            F(u, v-1)                     &  
            \end{tikzcd}
    \right) 
    \to F (u,v).
\end{equation}
This is a monomorphism by \autoref{lemma:me_square_equiv_cond}, since $F$ is middle-exact.

Hence, $\varepsilon_1 \colon T_1 F \to F$ is a cofibration in $\projPvec$. Thus, $C = \coker \varepsilon_1$ is cofibrant in $\projPvec$, and therefore $C$ is projective, and the sequence
\begin{equation*}
    T_1 F \to F \to C
\end{equation*}
splits. In other words, $F \cong T_1 F \oplus C$. By \autoref{lemma:projective_is_deg1}, $C$ is deg 1. In particular, $F$ is the direct sum of a degree 1 functor and a codegree 1 functor.
\end{example}

\subsection{The Reedy model structure and injective objects}

Everything in the previous section dualizes. We state the main definitions and results without proof, as the proofs dualize formally.

Let $\Fun_{\textrm{inj}}(P, \Vecs_{\Fb})$ denote the model structure on $\Fun(P, \Vecs_{\Fb})$ where the weak equivalences are all morphisms and the cofibrations are the objectwise epi natural transformations.

\begin{lemma}\label{lemma:restricting_fibrant}
    Let $P$ be a lattice satisfying the ascending chain condition, and $Q \subset P$ a finite join-sublattice.
    Suppose $F$ is an injective object in $\Fun(P, \Vecs_{\Fb})$. Then the restriction $F|_Q$ is an injective object in $\Fun(Q, \Vecs_\Fb)$.
\end{lemma}

\begin{lemma}\label{lemma:injective_is_codeg1}
    Let $P$ be a lattice satisfying the ascending chain condition.
    Let $F$ be an injective object in $\Fun(P, \Vecs_\Fb)$. Then $F$ is epimorphic and codegree 1.
\end{lemma}

\begin{lemma}\label{lemma:deg1_epi_is_inj}
    Let $P$ be an up-finite distributive lattice. If $F \colon P \to \Vecs_\Fb$ is epimorphic and degree 1, then $F$ is an injective object.
\end{lemma}

\subsection{Structure of injective and projective objects}

A persistence module $F \colon P \to \Vecs_{\Fb}$ is said to be \emph{free} if it is decomposable into interval modules of the form $\Fb_{a^\uparrow}$.
The following result is stated in \cite[Proposition 5]{projectiveDiagsFree}.
\begin{lemma}\label{lemma:projective_structure}
    Let $P$ be a poset and $\Fb$ a field.
    Projective objects in $\Fun(P, \Vecs_\Fb)$ are free.
\end{lemma}
It follows from \autoref{lemma:codeg1_mono_is_proj} that if $P$ is a finite distributive lattice and $F \colon P \to \Vecs_\Fb$ is monomorphic and codegree 1, then it is free.

Note that the dual to this lemma does not hold, as shown in \cite{injectiveDiags}. However, when $P$ satisfies the ascending chain condition (\autoref{def:descending_chain}), a similar result can be stated.

A persistence module $F \colon P \to \Vecs_{\Fb}$ is said to be \emph{T-standard} if it is isomorphic to a module of the form
\begin{equation*}
    \prod_{a \in I} \Fb_{a^\downarrow}.
\end{equation*}
The following result is stated in \cite[Proposition 2.2]{injectiveDiags}.
\begin{lemma}\label{lemma:injective_structure}
    Let $P$ be a poset satisfying the ascending chain condition, and let $\Fb$ be a field.
    Then every injective object in $\Fun(P, \Vecs_{\Fb})$ is T-standard.
\end{lemma}

In particular, when $P$ satisfies the ascending chain condition, every p.f.d., injective object in $\Fun(P, \Vecs_{\Fb})$ is interval decomposable.
Moreover, by \autoref{lemma:deg1_epi_is_inj}, if $P$ is a finite distributive lattice and $F \colon P \to \Vecs_\Fb$ is epimorphic and degree 1, then it is T-standard.

\section{Higher-dimensional multipersistence modules}

In this section, we move from the 2-parameter paradigm to more complicated posets.
We first study \emph{bidegree $n$} functors, i.e., functors that are both degree $n$ and codegree 1, and show that under certain conditions on the source poset target, $T_1$ ``preserves degree 1'' in the sense that it sends degree 1 functors to bidegree 1 functors (and dually).
We further study the \emph{layers} of the Taylor tower and show that the $n$th layer is bidegree $n$.

We then apply these result to prove several novel results on multipersistence modules whose source poset is a product of $n$ finite total orders, for $n \ge 2$.
We show that bidegree 1 p.f.d. multipersistence modules of these kind are always interval decomposable.
We further prove a theorem on the structure of these multiparameter persistence modules, which holds even when the modules are not p.f.d.
Specifically, we show that they are always isomorphic to the direct sum of an injective module, a projective module, and a bidegree 1 module.

\subsection{Theory of poset (co)calculus for functors into abelian categories}

\subsubsection{Bidegree 1 functors}

We now generalize \autoref{def:middle_exactness} to functors with source any lattice. We will use the word \emph{2-middle-exactness} in this case, to be consistent with the definitions introduced in \cite{lebovici2024localcharacterizationblockdecomposabilitymultiparameter}.

\begin{definition}\label{def:2_middle_exactness}
    Let $P$ be a lattice, and $\A$ an abelian category.
    A persistence module $F \colon P \to \A$ is \emph{2-middle-exact} if, for all $x, y \in P$, the square
    \begin{center}
    \begin{tikzcd}
    F(x \wedge y) \arrow[r] \arrow[d] & F(x) \arrow[d] \\
    F(y) \arrow[r]           & F(x \vee y)
    \end{tikzcd}
    \end{center}
    is middle-exact.
\end{definition}

\begin{lemma}\label{lemma:me_mono_is_deg1}
    Let $P$ be a lattice, $\A$ an abelian category, and $F \colon P \to \A$ 2-middle-exact.
    If $F$ is monomorphic, then $F$ is degree 1.

    Dually, if $F$ is epimorphic, then $F$ is codegree 1.
\end{lemma}

\begin{proof}
    This follows directly from \autoref{lemma:me_and_mono_is_pb}.
    %
\end{proof}

Recall that a subset $Q \subseteq P$ is \emph{down-closed} if for all $x \in Q$ and $y \in P$ such that $y \le x$, we have $y \in Q$.

\begin{lemma}\label{lemma:middle_exact_0_on_subposet}
    Let $\A$ be an abelian category, and let $P$ be a down-finite distributive lattice such that $P_{\le 1}$ is down-closed. Let $F \colon P \to \A$ be a 2-middle-exact functor such that $F|_{P_{\le 1}} \cong 0$. Then $F$ is monomorphic and degree 1.

    Dually, let $P$ be an up-finite distributive lattice such that $P^{\le 1}$ is up-closed. If $F \colon P \to \A$ is a 2-middle-exact functor such that $F|_{P^{\le 1}} \cong 0$, $F$ is epimorphic and codegree 1.
\end{lemma}

\begin{proof}
    We prove the first part. The second part is dual.

    Let $y \in P$, and set $Q = y^\downarrow$. We show that $F|_{Q}$ is monomorphic by induction.
    This will imply that $F$ is monomorphic, as $y$ was chosen arbitrarily.
    Let $x_1, \dots, x_N$ be an ordering of the elements in $Q$ such that $x_i \le x_j \implies i \le j$.

    First, observe that if $x \in Q_{\le 1}$, and $y \le x$, then $y \in Q_{\le 1}$ by assumption, so $F(y\le x)$ is simply the zero map $0 \to 0$, and thus a monomorphism.
    
    Suppose now that for all $i \le k$, we have that $F(y \le x_i)$ is a monomorphism for all $y \le x_i$, and suppose that $x_{k+1} \notin Q_{\le 1}$. We want to show that $F(y \le x_{k+1})$ is a monomorphism for all $y \le x_{k+1}$. Observe that it suffices to show this for $y \prec x_{k+1}$. By \autoref{lemma:num_of_parents} and \autoref{lemma:join_of_parents}, we can choose another $z \prec x_{k+1}$ such that $y \vee z = x_{k+1}$. Consider the middle-exact square
    \begin{center}
    \begin{tikzcd}
    F(y \wedge z) \arrow[r] \arrow[d] & F(z) \arrow[d] \\
    F(y) \arrow[r]           & F(x_{k+1}).
    \end{tikzcd}
    \end{center}
    The top arrow, $F(y \wedge z \le z)$ is a monomorphism, by the induction hypothesis. In particular, $\ker F(y \wedge z \le z) \cong 0$.
    Thus, $\ker F(y \le x_{k+1}) \cong 0$, by \autoref{lemma:me_square_coker_injective}.
    In other words, $F(y \le x_{k+1})$ is a monomorphism, as desired.

    We have shown that $F$ is monomorphic. Thus, by \autoref{lemma:me_mono_is_deg1}, $F$ is degree 1.
\end{proof}

\begin{definition}
    Let $P$ be a lattice and $\C$ a finitely bicomplete category. A functor $F \colon P \to \A$ is called \emph{bidegree $n$} if it is both degree $n$ and codegree $n$.
\end{definition}

\begin{remark}
    Bidegree 1 multipersistence modules are called \emph{2-exact} in \cite{lebovici2024localcharacterizationblockdecomposabilitymultiparameter}.
\end{remark}

\begin{proposition}\label{prop:T1_preserves_deg_1}
    Let $P$ be a down-finite distributive lattice such that $P_{\le 1}$ is down-closed, $\A$ an abelian category, and $F \colon P \to \A$ a degree 1 functor. Then $T_1 F$ is bidegree 1.

    Dually, if $P$ is an up-finite distributive lattice such that $P_{\ge 1}$ is down-closed, and $F \colon P \to \A$ is a codegree 1 functor, then $T^1 F$ is bidegree 1.
\end{proposition}

\begin{proof}
    We prove the first part. The second part is dual.

    Let $K = \ker(T_1 F \to F)$ and $C = \coker(T_1 F \to F)$. Note that $K|_{P_{\le 1}} \cong 0$ and $C|_{P_{\le 1}} \cong 0$. For $x,y \in P$, consider the following commuting diagram. 

    \begin{center}
        \begin{tikzcd}
            & K(x \wedge y) \arrow[r] \arrow[d]     & K(x) \oplus K(y) \arrow[r] \arrow[d]         & K(x \vee y) \arrow[d]               &   \\
            & T_1 F(x \wedge y) \arrow[r] \arrow[d] & T_1 F(x) \oplus T_1 F(y) \arrow[r] \arrow[d] & T_1 F(x \vee y) \arrow[r] \arrow[d] & 0 \\
0 \arrow[r] & F(x \wedge y) \arrow[r] \arrow[d]     & F(x) \oplus F(y) \arrow[r] \arrow[d]         & F(x \vee y) \arrow[d]               &   \\
            & C(x \wedge y) \arrow[r]               & C(x) \oplus C(y) \arrow[r]                   & C(x \vee y)                         &  
        \end{tikzcd}
    \end{center}

    By the snake lemma, both the top and bottom rows are exact, and so both $K$ and $C$ are 2-middle-exact functors.

    It now follows from \autoref{lemma:middle_exact_0_on_subposet} that $C$ is degree 1. Again applying the snake lemma to the diagram above, we see that $K$ is a codegree 1 functor. Thus, $K \cong T_1 K \cong 0$ (as $K|_{P_{\le 1}} = 0$).

    Finally, $T_1 F$ is now isomorphic to $\ker(F \to C)$, and is thus degree 1, as the limit of degree 1 functors is degree 1.
\end{proof}

\begin{example}\label{ex:down_closed}
    Examples of posets $P$ where $P_{\le 1}$ is down-closed include: products of total orders, the poset $(\Po(V), \subseteq)$ for any set $V$, and the poset $(\Nn_{> 0}, |)$ of positive integers with the total order given by the divisibility relation. In these posets, $P^{\le 1}$ is also up-closed.

    Non-examples include the following.
    Let $P$ be the following poset, which is a finite distributive lattice.
    \begin{center}
    \begin{tikzcd}
                       & t                       &                    \\
                       & d \arrow[u]             &                    \\
    b \arrow[ru] &                               & c \arrow[lu] \\
                       & a \arrow[lu] \arrow[ru] &                   
    \end{tikzcd}
    \end{center}
    Here, $t$ is join-irreducible, but $d < t$ is not. Hence, $P_{\le 1}$ is not down-closed.

\end{example}

\subsubsection{Layers}

Given a pointed category $\C$ with basepoint $*$, and a morphism $f \colon A \to B$ in $\C$, the \emph{fiber} of $f$, denoted $\fib(f)$, is defined as the limit 
\begin{equation*}
    \lim(A \xrightarrow{f} B \leftarrow *),
\end{equation*}
if it exists.
Dually, the \emph{cofiber} of $f$, denoted $\cofib(f)$, is defined as the colimit \begin{equation*}
    \colim(B \xleftarrow{f} A \to *),
\end{equation*}
if it exists.

In the following definition, observe that if $P$ is a distributive lattice with minimal element 0, then by \autoref{eq:T_n_cat}, $T_0 F$ is the constant functor at $F(0)$ (and dually, $T^0F$ is the constant functor at the maximal element).

\begin{definition}
    Let $\C$ be a complete pointed category, let $P$ be a distributive lattice, and let $n$ be a nonnegative integer. The $n$th \emph{layer} of $F$, denoted $D^n F$ is defined as the fiber
    \begin{equation*}
        D^n F = \fib(T^n F \to T^{n-1} F).
    \end{equation*}

    Dually, if $\C$ is a cocomplete pointed category, the $n$th \emph{colayer} of $F$ is the cofiber
    \begin{equation*}
        D_n F = \cofib(T_{n-1} F \to T_n F).
    \end{equation*}
\end{definition}

\begin{proposition}\label{prop:layers}
    Let $P$ be a finite product of total orders with minimal elements, $\C$ a complete pointed category, and $F \colon P \to \C$ any functor. Then $D_n F$ is bidegree $n$.

    Dually, let $P$ be a finite product of total orders with maximal elements, $\C$ a cocomplete pointed category, and $F \colon P \to \C$ any functor. Then $D^n F$ is bidegree $n$.
\end{proposition}

\begin{proof}
    We give the proof for $D_n F$. The proof for $D^n F$ is dual.

    Let $0$ denote the basepoint in $\C$. Let $P = S_0 \times \dots \times S_k$, where each $S_i$ is a total order with minimal element $0_i$.
    For $S \subseteq [k]$, let $\lambda_S \colon P \to P$ be given by
    \begin{equation*}
        (\lambda_S ( x_0, \dots, x_k))_i = \begin{cases}
            x_i, \quad i \in S, \\
            0_i, \quad \textrm{otherwise.}
        \end{cases}
    \end{equation*}
    We first show that
    \begin{equation*}
        D_n F \cong \bigoplus_{S \subseteq [k], |S| = n} (D_n F) \circ \lambda_S.
    \end{equation*}
    
    We know that $D_n F$ is codegree $n$, since it is the cokernel of a natural transformation between codegree $n$ functors. Moreover, $D_n F$ is 0 on $P_{\le n-1}$. Hence, for $x = (x_1, \dots, x_k) \in P$,
    \begin{align*}
        D_n F (x) &= \underset{y \in P_{\le n}, y \le x}{\colim \ } D_n F(y) \cong \underset{S \subseteq [k], |S| \le n}{\colim} D_n F (\lambda_S(x)). 
    \end{align*}
    The last isomorphism follows from \cite[Remark 5.1]{hem2025posetfunctorcocalculusapplications}.
    Now, as $D_n F(\lambda_S(x)) \cong 0$ for $|S| \le n-1$ (as then $\lambda_S(x) \in P_{\le n-1}$), $D_n F (x)$ is isomorphic to
    \begin{equation*}
        \bigoplus_{S \subseteq [k], |S| = n} D_n F (\lambda_S(x)).
    \end{equation*}
    Hence, the canonical map
    \begin{equation*}
        \bigoplus_{S \subseteq [k], |S| = n} (D_n F)\circ \lambda_S \to D_n F
    \end{equation*}
    is an isomorphism.

    Now, as $T^n$ is an additive functor, it suffices to show that $(D_n F) \circ \lambda_S$ is degree $n$ for each $S \subseteq [k]$ with $|S| = n$. This follows directly from \cite[Proposition 3.15]{hem2025posetfunctorcocalculusapplications}, so we are done.
\end{proof}

\subsection{Application: Higher-dimensional middle exactness}

\subsubsection{Interval decomposition of bidegree 1 functors}

\begin{proposition}\label{prop:bideg1_interval_decomp}
    Let $P$ be a finite product of finite total orders, $\Fb$ a field, and $F \colon P \to \Vecs_{\Fb}$ a p.f.d. bidegree 1 functor. Then $F$ is interval decomposable.
\end{proposition}

\begin{proof}
    Let $0$ denote the minimal element in $P$.
    As $F$ is codegree 1, $F \cong T_1 F$. Thus, $D_1 F = \coker(F(0) \to T_1F) \cong \coker(F(0) \to F)$. Let $K = \Ima \left( F(0) \to F\right)$, and consider the short exact sequence
    \begin{equation*}
        0 \to K \to F \to D_1 F \to 0.
    \end{equation*}
    By \autoref{prop:layers}, $D_1 F$ is degree 1, so $K$ is the kernel of degree 1 functors, and thus degree 1. Furthermore, it follows from the definition of $K$ that $K$ is epimorphic.
    Hence, $K$ is injective by \autoref{lemma:deg1_epi_is_inj}.
    Thus, $F \cong K \oplus D_1 F$, and $K$ is interval decomposable by \autoref{lemma:injective_structure}.

    It remains to show that $D_1 F$ is interval decomposable.
    Write $P = S_0 \times \dots \times S_N$, where each $S_i$ is a finite total order.
    Observe that by \autoref{prop:layers},
    \begin{equation*}
        D_n F \cong \bigoplus_{i \in [N]} D_n F \circ \lambda_{\{i\}}.
    \end{equation*}
    Observe further that $D_n F \circ \lambda_{\{i\}} \cong D_n F \circ j_i \circ \pi_i$, where $\pi_i \colon P \to S_i$ is the projection map and $j_i \colon S_i \to P$ is the inclusion given by
    \begin{equation*}
        j_i(u)_l = \begin{cases}
            u, \quad l = i \\
            0, \quad \textrm{otherwise.}
        \end{cases}
    \end{equation*}
    Now, $D_n F \circ j_i$ is interval decomposable as it is a p.f.d. persistence module from a finite total order, by the structure theorem of persistence modules. 
    Hence, $D_n F \circ \lambda_{\{i\}}$ is interval decomposable, which concludes the proof. 
\end{proof}

\begin{example}[Non-example]
    The following functor into $\Vecs_\Fb$, defined on the poset $P$ from \autoref{ex:down_closed}, is bidegree 1, but not interval decomposable.
    \begin{center}
    \begin{tikzcd}
                       & \Fb                       &                    \\
                       & \Fb^2 \arrow[u, "\begin{pmatrix}1 \ 1\end{pmatrix}"]             &                    \\
    \Fb \arrow[ru, "\begin{pmatrix}1 \\ 0\end{pmatrix}"] &                               & \Fb \arrow[lu, "\begin{pmatrix}0 \\ 1\end{pmatrix}"'] \\
                       & 0 \arrow[lu] \arrow[ru] &                   
    \end{tikzcd}
    \end{center}
    To see that it's bidegree 1, observe that the only nontrivial bicartesian square is the one consisting of the lower four elements $\{a,b,c,d\}$ (this is bicartesian as $a = b \wedge c$ and $d = b \vee c$). It suffices to check that the functor restricted to this square is both a pullback and pushout. In other words, one needs to check that the square
    \begin{center}
    \begin{tikzcd}
    \Fb \arrow[r, "\begin{pmatrix}1 \\ 0\end{pmatrix}"]    & \Fb^2               \\
    0 \arrow[u] \arrow[r] & \Fb \arrow[u, "\begin{pmatrix}0 \\ 1\end{pmatrix}"']
    \end{tikzcd}
    \end{center}
    is a pushout and a pullback, which is easily verified.
\end{example}

\begin{example}[Non-example]
We give an example showing that the \emph{bidegree 1} condition in \autoref{prop:bideg1_interval_decomp} cannot be relaxed to codegree 1 (or degree 1).
The following diagram, from \cite[Example 15]{lebovici2024localcharacterizationblockdecomposabilitymultiparameter}, shows an indecomposable multipersistence module $\{0,1\}^3 \to \Vecs_{\Fb}$.
\begin{center}
\begin{tikzcd}
& 0 \arrow[rr] \arrow[from=dd]              &                           & 0            \\
\Fb \arrow[rr, crossing over] \arrow[ru]                                               &                           & 0 \arrow[ru]              &              \\
& \Fb \arrow[rr] &                           & 0 \arrow[uu] \\
\Fb^2 \arrow[uu, "(0 \ 1)"] \arrow[rr, "(1 \ 1)"] \arrow[ru, "(1 \ 0)"] &                           & \Fb \arrow[uu, crossing over] \arrow[ru] &             
\end{tikzcd}
\end{center}
Every square in this cube is a pushout. Hence, the multipersistence module is a codegree 1 functor.

One can construct a similar example of an indecomposable \emph{degree 1} multipersistence module, by inverting the arrows and transposing the vectors in the above example.
\end{example}

\subsubsection{Higher-dimensional middle exactness}

We recall here the usual definition of the Koszul complex.
\begin{definition}
    Given a $k$-cube $\X \colon \Po([k-1]) \to \Vecs_{\Fb}$, the \emph{Koszul complex} of $\X$, denoted $K_{\X}$, is the chain complex in $\Vecs_{\Fb}$ given by
    \begin{equation*}
        (K_{\X})_i = \bigoplus_{S \subseteq [k], |S| = k-i} \X (S),
    \end{equation*}
    with differential $\partial_{i+1} \colon (K_{\X})_{i+1} \to (K_{\X})_{i}$ defined componentwise by
    \begin{equation*}
        \partial_{i+1}|_{\X(S)}
        = \sum_{j=0}^i (-1)^j \X (S \subseteq S \cup \{t_j\}),
    \end{equation*}
    where $t_0 < \dots < t_i$ are the elements in $[k] \setminus S$.
\end{definition}

We introduce the following definition, which generalizes that of $k$-middle-exactness in \cite{lebovici2024localcharacterizationblockdecomposabilitymultiparameter}.

\begin{definition}
    Let $P$ be a distributive lattice and $\Fb$ a field. A functor $F \colon P \to \Vecs_{\Fb}$ is said to be \emph{$k$-middle-exact} if for every strongly bicartesian $k$-cube $\X$ in $P$, the Koszul complex of $F \circ \X$ has trivial homology in all degrees $0 < i < k$.
\end{definition}

\begin{lemma}\label{lemma:cube_colimits_middle_exact}
    Let $F \colon \Po([k-1]) \to \Vecs_{\Fb}$ be a $k$-cube. Then
    \begin{equation*}
        \underset{{S \subsetneq [k-1]}}{\colim} \X(S) \cong \coker\left((K_\X)_2 \xrightarrow{\partial_2} (K_\X)_1\right),
    \end{equation*}
    and
    \begin{equation*}
        \lim_{\emptyset \subsetneq S \subseteq [k-1]} \X(S) \cong \ker\left((K_\X)_{k-2} \xrightarrow{\partial_{k-2}} (K_\X)_{k-3}\right).
    \end{equation*}
\end{lemma}
\begin{proof}
    We prove the first isomorphism. The second is dual.
    By finality of the inclusion of the indexing posets, we have that,
    \begin{equation*}
        \underset{{S \subsetneq [k-1]}}{\colim} \X(S) \cong \underset{{S \subsetneq [k-1]}, \\ k-3 \le |S| \le k-2}{\colim} \X(S).
    \end{equation*}
    
    Unwinding the definition of the Koszul complex, we get
    \begin{align*}
        &\coker\left((K_\X)_2 \xrightarrow{\partial_2} (K_\X)_1\right) \\
        \cong &\coker\left(
            \bigoplus_{0 \le i_0 < i_1 \le k-1}\X([k] \setminus \{i_0, i_1\})
            \xrightarrow{\partial_2}
            \bigoplus_{0 \le j \le k-1}\X([k] \setminus \{j\})
        \right),
    \end{align*}
    where
    \begin{align*}
        &\partial_{2}|_{\X([k-1] \setminus \{i_0,i_1\})} \\
        = &\X(([k-1] \setminus \{i_0,i_1\}) \subseteq ([k-1] \setminus \{i_1\}))
        - \X(([k-1] \setminus \{i_0,i_1\}) \subseteq ([k-1] \setminus \{i_0\})).
    \end{align*}
    
    It follows from checking universal properties that
    \begin{equation*}
        \underset{{S \subsetneq [k-1]},\\ k-3 \le |S| \le k-2}{\colim} \X(S) \cong \coker\left((K_\X)_2 \xrightarrow{\partial_2} (K_\X)_1\right).
    \end{equation*}
\end{proof}

Let $P$ be a finite distributive lattice, and let $x \in P$. Let $u_0, \dots, u_{n-1}$ denote the parents of $x$. By \autoref{lemma:join_of_parents}, $\{u_0, \dots, u_{n-1}\}$ is a pairwise cover of $x$. By \autoref{lem:bicartesian_from_codecomp}, this pairwise cover gives rise to a strongly bicartesian $n$-cube $\X_{u_0, \dots, u_{n-1}}$, defined by
\begin{equation*}
    \X_{u_0, \dots, u_{n-1}} (S) = 
    \begin{cases}
        x, &\quad S = [n-1], \\ 
        \bigwedge_{i \notin S} u_i, &\quad \text{ otherwise.} \\
    \end{cases}
\end{equation*}

\begin{lemma}\label{lemma:middle_exact_cofibration}
    Let $P$ be a finite distributive lattice such that $P_{\le 1}$ is down-closed, and let $\Fb$ be a field. If $F \colon P \to \Fb$ is a functor that is $k$-middle-exact for all $k \ge 2$, then the canonical morphism
    \begin{equation*}
        \varepsilon_1 \colon T_1 F \to F
    \end{equation*}
    is a split monomorphism with projective cokernel.
    
    Dually, let $P$ be a finite distributive lattice such that $P^{\le 1}$ is up-closed, and let $\Fb$ be a field. If $F \colon P \to \Fb$ is a functor that is $k$-middle-exact for all $k \ge 2$, then the canonical morphism
    \begin{equation*}
        \eta^1 \colon F \to T^1 F
    \end{equation*}
    is a split epimorphism with injective kernel.
\end{lemma}
\begin{proof}
    We prove the first statement. The second is formally dual.

    It is sufficient to show that $\varepsilon_1$ is a cofibration in $\projPvec$. We need to show that the latching map \eqref{eq:latching_map} is a monomorphism for each $x \in P$.

    First, suppose $x$ is the minimal element in $P$, which has join-dimension 0. The latching map here is
    \begin{equation*}
        T_1 F(x) \to F(x),
    \end{equation*}
    which is an isomorphism.

    Now, suppose that $\jdim(x) = 1$, and let $y \prec x$. Observe that then, by assumption, $y \in P_{\le 1}$. The latching map at $x$ is, by \autoref{lemma:latching_ob_cube_dim1},
    \begin{equation*}
        F (y) \coprod_{T_1 F(y)} T_1 F(x) \to F(x),
    \end{equation*}
    which is an isomorphism as $T_1 F (y) = F(y)$, and $T_1 F(x) = F(x)$.

    Now, let $x \in P$ with $\jdim(x) \ge 2$. First, observe that
    \begin{equation*}
        L_x (T_1 F) \to T_1 F(x)
    \end{equation*}
    is an isomorphism, by \autoref{lemma:latching_ob_cube}.
    Let now $u_0, \dots, u_{n-1}$ be the parents of $x$, and let $\X = \X_{u_1, \dots, u_n}$.
    By \autoref{lemma:latching_ob_cube} and \autoref{lemma:cube_colimits_middle_exact}, the latching map $L_x F \to F(x)$ is the map
    \begin{equation*}
        \coker\left((K_\X)_2 \xrightarrow{\partial_2} (K_\X)_1\right) \ \to \ (K_{\X})_0.
    \end{equation*}
    Now, by definition, $H_1(K_\X)$ is precisely the kernel of this map. As we assumed $F$ to be $n$-middle-exact, $H_1(K_\X) \cong 0$, and so the latching map is a monomorphism.
    This concludes the proof.
\end{proof}

\begin{remark}
If $P$ is a poset that
satisfy both conditions of \autoref{lemma:middle_exact_cofibration},
i.e., if $P$ is a finite distributive lattice such that $P_{\le 1}$ is down-closed and $P^{\le 1}$ is up-closed,
then $P$ is a finite product of finite total orders.
We omit the details of the proof here, and just mention the main points of the argument.

Because every element in a finite distributive lattice has a unique reduced indecomposable join-decomposition, the structure of $P$ is entirely decided by the structure of $P_{\le 1}$. Supposing now that $P$ is not a product of total orders, there must be elements $x,y,z \in P_{\le 1}$ with $x < y \wedge z < y,z$. Considering the element
\begin{equation*}
    p = \bigvee_{a \in P, a \ngeq y \wedge z} a,
\end{equation*}
leads to a contradiction to the assumption that $P^{\le 1}$ is meet-closed.
\end{remark}

\begin{theorem}\label{thm:middle_exactnes_general}
    Let $P$ be a finite product of finite total orders, and let $\Fb$ be a field. Let $F \colon P \to \Vecs_{\Fb}$ be a functor that is $k$-middle-exact for all $k \ge 2$.
    
    Then there exist $B,K,C \colon P \to \Vecs_{\Fb}$ with $K$ injective, $C$ projective and $B$ bidegree 1, such that $F \cong B \oplus K \oplus C$.
\end{theorem}
\begin{proof}
    Let $K = \ker(F \to T^1F)$, and $C = \coker(T_1 F \to F)$.
    By \autoref{lemma:middle_exact_cofibration}, $C$ is projective, and the sequence
    \begin{equation*}
        0 \to T_1 F \to F \to C \to 0
    \end{equation*}
    splits, so $F \cong T_1 F \oplus C$.
    Furthermore, by the same lemma, $K$ is injective and the sequence
    \begin{equation*}
        0 \to K \to F \to T^1 F \to 0
    \end{equation*}
    is split exact, and so by additivity of $T_1$, the sequence
    \begin{equation*}
        0 \to T_1 K \to T_1 F \to T_1 T^1 F \to 0
    \end{equation*}
    is also split exact. By \autoref{lemma:injective_is_codeg1}, $T_1 K \cong K$, and so
    \begin{equation*}
        F \cong T_1 F \oplus C \cong T_1 T^1 F \oplus K \oplus C.
    \end{equation*}
    Applying \autoref{prop:T1_preserves_deg_1} concludes the proof.
\end{proof}

\begin{example}[Non-example]
    The following functor into $\Vecs_\Fb$, defined on the poset $P$ from \autoref{ex:down_closed}, is $k$-middle-exact for all $k \ge 2$, and is indecomposable, but it is neither projective, injective nor bidegree 1.
    \begin{center}
    \begin{tikzcd}
                       & 0                       &                    \\
                       & \Fb \arrow[u]             &                    \\
    0 \arrow[ru] &                               & 0 \arrow[lu] \\
                       & 0 \arrow[lu] \arrow[ru] &                   
    \end{tikzcd}
    \end{center}
    To see that it's $k$-middle-exact for all $k$, observe that the only nontrivial bicartesian $k$-cube, for $k \ge 2$, is the 2-cube (i.e., square) consisting of the lower four elements $\{a,b,c,d\}$. It suffices to check that the associated complex of the functor at this square is middle-exact. In other words, one needs to check that the complex
    \begin{equation*}
        0 \to 0 \oplus 0 \to \Fb
    \end{equation*}
    is middle-exact, which is easily verified.
\end{example}

\begin{corollary}
    Let $P = S_1 \times \dots \times S_N$ be a finite product of finite total orders, $\Fb$ a field, and $F \colon P \to \Vecs_{\Fb}$ a p.f.d. functor that is $k$-middle-exact for all $2 \le k \le N$. Then $F$ is interval decomposable.
\end{corollary}
\begin{proof}
    By \autoref{thm:middle_exactnes_general}, $F \cong B \oplus K \oplus C$ with $B$ bidegree 1, $K$ injective and $C$ projective. Now, $C$ is interval decomposable by \autoref{lemma:projective_structure}, $K$ is interval decomposable by \autoref{lemma:injective_structure} (as $K$ is p.f.d.) and $B$ is interval decomposable by \autoref{prop:bideg1_interval_decomp} (as $B$ is p.f.d.).
\end{proof}

\bibliographystyle{plain}
\bibliography{ref}

\begin{thebibliography}{10}

\bibitem{edelsbrunner}
P.~Bendich, H.~Edelsbrunner, and M.~Kerber.
\newblock Computing robustness and persistence for images.
\newblock {\em IEEE transactions on visualization and computer graphics}, 16(6):1251--1260, 2010.

\bibitem{Birkhoff}
G.~Birkhoff.
\newblock {\em Lattice {T}heory}, volume Vol. 25 of {\em American Mathematical Society Colloquium Publications}.
\newblock American Mathematical Society, New York, revised edition, 1948.

\bibitem{botnanMiddleExactness}
M.B. Botnan and W.~Crawley-Boevey.
\newblock Decomposition of persistence modules.
\newblock {\em Proc. Amer. Math. Soc.}, 148(11):4581--4596, 2020.

\bibitem{BotnanLesnickMultipersistence}
M.B. Botnan and M.~Lesnick.
\newblock An introduction to multiparameter persistence.
\newblock In {\em Representations of algebras and related structures}, EMS Ser. Congr. Rep., pages 77--150. EMS Press, Berlin, [2023] \copyright 2023.

\bibitem{interlevelset_persistence_carlsson}
G.~Carlsson, V.~de~Silva, and D.~Morozov.
\newblock Zigzag persistent homology and real-valued functions.
\newblock In {\em Proceedings of the Twenty-Fifth Annual Symposium on Computational Geometry}, SCG '09, page 247–256, New York, NY, USA, 2009. Association for Computing Machinery.

\bibitem{interval_decomp_source}
G.~Carlsson, A.~Zomorodian, A.~Collins, and L.~Guibas.
\newblock Persistence barcodes for shapes.
\newblock {\em International Journal of Shape Modeling}, 11:149--188, 01 2005.

\bibitem{realisationsposetstameness}
W.~Chacholski, A.~Jin, and F.~Tombari.
\newblock Realisations of posets and tameness, 2024.

\bibitem{structure_theorem_crawley}
W.~Crawley-Boevey.
\newblock Decomposition of pointwise finite-dimensional persistence modules.
\newblock {\em J. Algebra Appl.}, 14(5):1550066, 8, 2015.

\bibitem{hem2025posetfunctorcocalculusapplications}
B.G. Hem.
\newblock Poset functor cocalculus and applications to topological data analysis, 2025.
\newblock Preprint available on \href{https://arxiv.org/abs/2501.05996}{arXiv:2501.05996}.

\bibitem{Hess_2017}
K.~Hess, M.~K{\c{e}}dziorek, E.~Riehl, and B.~Shipley.
\newblock A necessary and sufficient condition for induced model structures.
\newblock {\em Journal of Topology}, 10(2):324–369, April 2017.

\bibitem{injectiveDiags}
M.~H{\"o}ppner.
\newblock A note on the structure of injective diagrams.
\newblock {\em Manuscripta Mathematica}, 44(1):45--50, 1983.

\bibitem{hovey}
M.~Hovey.
\newblock {\em Model categories}, volume~63 of {\em Mathematical Surveys and Monographs}.
\newblock American Mathematical Society, Providence, RI, 1999.

\bibitem{projectiveDiagsFree}
M.~Höppner and H.~Lenzing.
\newblock Projective diagrams over partially ordered sets are free.
\newblock {\em Journal of Pure and Applied Algebra}, 20(1):7--12, 1981.

\bibitem{lebovici2024localcharacterizationblockdecomposabilitymultiparameter}
V.~Lebovici, J.-P. Lerch, and S.~Oudot.
\newblock Local characterization of block-decomposability for multiparameter persistence modules, 2024.

\bibitem{lerch_thesis}
J.-P. Lerch.
\newblock {\em On the representation theory of persistence modules}.
\newblock PhD thesis, Universität Bielefeld, 2023.

\bibitem{lurie2017higher}
J.~Lurie.
\newblock Higher algebra.
\newblock Unpublished. Available online at \url{https://www.math.ias.edu/~lurie/papers/HA.pdf}, 09 2017.

\bibitem{structure_theorem_webb}
C.~Webb.
\newblock Decomposition of graded modules.
\newblock {\em Proc. Amer. Math. Soc.}, 94(4):565--571, 1985.

\end{thebibliography}

\end{document}